\documentclass[12pt,reqno]{amsart}

\usepackage{amssymb,amsmath, amsfonts, latexsym}
\usepackage{enumerate}
\usepackage{url}
\newtheorem{thm}{Theorem}[section]
\newtheorem{cor}[thm]{Corollary}
\newtheorem{lem}[thm]{Lemma}

\newtheorem{rem}{Remark}[section]

\newcommand{\dE}{\mathbb{E}}
\newcommand{\dF}{\mathbb{F}}
\newcommand{\dR}{\mathbb{R}}
\newcommand{\dP}{\mathbb{P}}

\newcommand{\cF}{\mathcal{F}}
\newcommand{\cN}{\mathcal{N}}

\newcommand{\veps}{\varepsilon}
\newcommand{\wh}{\widehat}
\newcommand{\tn}{\wh{\theta}_{n}}
\newcommand{\rn}{\wh{\rho}_{n}}
\newcommand{\dn}{\wh{D}_{n}}
\newcommand{\en}{\wh{\veps}_{n}}
\newcommand{\ek}{\wh{\veps}_{k}}
\newcommand{\eek}{\wh{\veps}_{k-1}}

\newcommand{\superexp} {\underset{b_n^2}{\overset{\rm superexp}{\longrightarrow}} }
\newcommand{\superexpequiv} {\underset{b_n^2}{\overset{\rm superexp}{\sim}} }
\newcommand{\superexpn} {\underset{n}{\overset{\rm superexp}{\longrightarrow}} }

\def\G#1{\textup{\textbf{(G.\ref{#1})}}}
\def\CL#1{\textup{\textbf{(CL.\ref{#1})}}}

\font\calcal=cmsy10 scaled\magstep1
\def\build#1_#2^#3{\mathrel{\mathop{\kern 0pt#1}\limits_{#2}^{#3}}}
\def\liml{\build{\longrightarrow}_{}^{{\mbox{\calcal L}}}}

 \numberwithin{equation}{section}

\begin{document}

\title[Moderate deviations for the Durbin-Watson statistic]{Moderate deviations for the Durbin-Watson statistic related to the first-order autoregressive process}

\author{S.Val\`ere Bitseki Penda}
\email{Valere.Bitsekipenda@math.univ-bpclermont.fr}
\address{\vspace{-1cm}}

\author{Hac\`ene Djellout}
\email{Hacene.Djellout@math.univ-bpclermont.fr}
\address{Laboratoire de Math\'ematiques, CNRS UMR 6620, Universit\'e Blaise Pascal, Avenue des Landais, 63177 Aubi\`ere, France.}

\author{Fr\'ed\'eric Pro\"ia}
\email{Frederic.Proia@inria.fr}
\address{Universit\'e Bordeaux 1, Institut de Math\'ematiques de Bordeaux,
UMR 5251, and INRIA Bordeaux, team ALEA, 351 Cours de la Lib\'eration, 33405 Talence cedex, France.}


\begin{abstract} The purpose of this paper is to investigate moderate deviations for the Durbin-Watson statistic associated with the stable first-order autoregressive process where the driven noise is also given by a first-order autoregressive process. We first establish a moderate deviation principle for both the least squares estimator of the unknown parameter of the autoregressive process as well as for the serial correlation estimator associated with the driven noise. It enables us to provide a moderate deviation principle for the Durbin-Watson statistic in the easy case where the driven noise is normally distributed and in the more general case where the driven noise satisfies a less restrictive Chen-Ledoux type condition.

\end{abstract}

\maketitle

\vspace{-0.5cm}

\begin{center}
\end{center}

\medskip

\maketitle

\section{Introduction}

This paper is focused on the stable first-order autoregressive process where the driven noise is also given by a first-order autoregressive process. The purpose is to investigate moderate deviations for both the least squares estimator of the unknown parameter of the autoregressive process as well as for the serial correlation estimator associated with the driven noise. Our goal is to establish moderate deviations for the Durbin-Watson statistic \cite{DurbinWatson50}, \cite{DurbinWatson51}, \cite{DurbinWatson71}, in a lagged dependent random variables framework. First of all, we shall assume that the driven noise is normally distributed. Then, we will extend our investigation to the more general framework where the driven noise satisfies a less restrictive Chen-Ledoux type condition \cite{Chen98}, \cite{Ledoux92}. We are inspired by the recent paper of Bercu and Pro\"ia \cite{BercuProia11}, where the almost sure convergence and the central limit theorem are established for both the least squares estimators and the Durbin-Watson statistic. Our results are proved via an extensive use of the results of Dembo \cite{Dembo96}, Dembo and Zeitouni \cite{DemboZeitouni98} and Worms \cite{Worms00}, \cite{Worms01a}, \cite{Worms01b} on the one hand, and of the paper of Puhalskii \cite{Puhalskii97} and Djellout \cite{Djellout02} on the other hand, about moderate deviations for martingales. In order to introduce the Durbin-Watson statistic, we shall focus our attention on the first-order autoregressive process given, for all $n\geq 1$, by
\begin{equation}
\label{AR}
\vspace{1ex}
\left\{
\begin{array}[c]{ccccc}
X_{n} & = & \theta X_{n-1} & + & \veps_n \vspace{1ex}\\
\veps_n & = & \rho \veps_{n-1} & + & V_n
\end{array}
\right.
\end{equation}
where we shall assume that the unknown parameters $\vert \theta \vert < 1$ and $\vert \rho \vert < 1$ to ensure the stability of the model. In all the sequel, we also assume that $(V_{n})$ is a sequence of independent and identically distributed random variables with zero mean, positive variance $\sigma^2$ and satisfying some suitable assumptions. The square-integrable initial values $X_0$ and $\veps_0$ may be arbitrarily chosen. We have decided to estimate $\theta$ by the least squares estimator
\begin{equation}
\label{Est_Theta}
\tn = \frac{\sum_{k=1}^{n} X_{k} X_{k-1}}{\sum_{k=1}^{n} X_{k-1}^2}.
\end{equation}
Then, we also define a set of least squares residuals given, for all $1 \leq k \leq n$, by
\begin{equation}
\label{Est_Res}
\ek = X_{k} - \tn X_{k-1},
\end{equation}
which leads to the estimator of $\rho$,
\begin{equation}
\label{Est_Rho}
\rn = \frac{\sum_{k=1}^{n} \ek \eek}{\sum_{k=1}^{n} \eek^{\ 2}}.
\end{equation}
Finally, the Durbin-Watson statistic is defined, for $n \geq 1$, as
\begin{equation}
\label{Est_D}
\dn = \frac{\sum_{k=1}^{n} (\ek - \eek)^2}{\sum_{k=0}^{n} \ek^{\ 2}}.
\end{equation}
This well-known statistic was introduced by the pioneer work of Durbin and Watson \cite{DurbinWatson50}, \cite{DurbinWatson51}, \cite{DurbinWatson71}, in the middle of last century, to test the presence of a significative first order serial correlation in the residuals of a regression analysis. A wide range of litterature is available on the asymptotic behavior of the Durbin-Watson statistic, frequently used in Econometry. While it appeared to work pretty well in the classical independent framework, Malinvaud \cite{Malinvaud61} and Nerlove and Wallis \cite{NerloveWallis66} observed that, for linear regression models containing lagged dependent random variables, the Durbin-Watson statistic may be asymptotically biased, potentially leading to inadequate conclusions. Durbin \cite{Durbin70} proposed alternative tests to prevent this misuse, such as the \textit{h-test} and the \textit{t-test}, then substantial contributions were brought by Inder \cite{Inder86}, King and Wu \cite{KingWu91} and more recently Stocker \cite{Stocker06}. Lately, a set of results have been established by Bercu and Pro\"ia in \cite{BercuProia11}, in particular a test procedure as powerful as the \textit{h-test}, and they will be summarized thereafter as a basis for this paper.

\medskip

The paper is organized as follows. First of all, we recall the results recently established by Bercu and Pro\"ia \cite{BercuProia11}. In Section 2, we propose moderate deviation principles for the estimators of $\theta$ and $\rho$ and for the Durbin-Watson statistic, given by \eqref{Est_Theta}, \eqref{Est_Rho} and \eqref{Est_D}, under the normality assumption on the driven noise. Section 3 deals with the generalization of the latter results under a less restrictive Chen-Ledoux type condition on $(V_{n})$. Finally, all technical proofs are postponed to Section 4.

\medskip

\begin{lem}
\label{LEM_CVG_TLC_THETA}
We have the almost sure convergence of the autoregressive estimator,
\begin{equation*}
\lim_{n\rightarrow \infty} \tn = \theta^{*} \hspace{0.5cm} \textnormal{a.s.}
\end{equation*}
where the limiting value
\begin{equation}
\label{LimT}
\theta^{*} = \frac{\theta + \rho}{1 + \theta\rho}.
\end{equation}
In addition, as soon as $\dE[V_1^4] < \infty$, we also have the asymptotic normality,
\begin{equation*}
\sqrt{n} \left( \tn - \theta^{*} \right) \liml \cN \big( 0, \sigma^2_{\theta} \big)
\end{equation*}
where the asymptotic variance
\begin{equation}
\label{SigT}
\sigma^2_{\theta} = \frac{(1-\theta^2)(1-\theta\rho)(1-\rho^2)}{(1+\theta\rho)^3}.
\end{equation}
\end{lem}

\medskip

\begin{lem}
\label{LEM_CVG_TLC_RHO}
We have the almost sure convergence of the serial correlation estimator,
\begin{equation*}
\lim_{n\rightarrow \infty} \rn = \rho^{*} \hspace{0.5cm} \textnormal{a.s.}
\end{equation*}
where the limiting value
\begin{equation}
\label{LimR}
\rho^{*} = \theta \rho \theta^{*}.
\end{equation}
Moreover, as soon as $\dE[V_1^4] < \infty$, we have the asymptotic normality,
\begin{equation*}
\sqrt{n} \Big( \rn - \rho^{*} \Big) \liml \cN \big( 0, \sigma^2_{\rho} \big)
\end{equation*}
with the asymptotic variance
\begin{equation}
\label{SigR}
\sigma^2_{\rho} = \frac{(1-\theta \rho)}{(1+\theta\rho)^3}\big((\theta+\rho)^2(1+\theta \rho)^2+(\theta \rho)^2 (1-\theta^2)(1-\rho^2)\big).
\end{equation}
On top of that, we have the joint asymptotic normality,
\begin{equation*}
\sqrt{n} \begin{pmatrix}
\ \tn - \theta^{*}  \\
\ \rn - \rho^{*}
\end{pmatrix}
 \liml \cN \big( 0, \Gamma \big)
\end{equation*}
where the covariance matrix
\begin{equation}
\label{Gamma}
\Gamma = \begin{pmatrix}
\sigma^2_{\theta} & \ \theta \rho \sigma^2_{\theta} \\
\theta \rho \sigma^2_{\theta} & \sigma^2_{\rho}
\end{pmatrix}.
\end{equation}
\end{lem}

\medskip

\begin{lem}
\label{LEM_CVG_TLC_D}
We have the almost sure convergence of the Durbin-Watson statistic,
\begin{equation*}
\lim_{n\rightarrow \infty} \dn = D^{*} \hspace{0.5cm} \textnormal{a.s.}
\end{equation*}
where the limiting value
\begin{equation}
\label{LimD}
D^{*} = 2(1 - \rho^{*}).
\end{equation}
In addition, as soon as $\dE[V_1^4] < \infty$, we have the asymptotic normality,
\begin{equation*}
\sqrt{n} \left( \dn - D^{*} \right) \liml \cN \big( 0, \sigma^2_{D} \big)
\end{equation*}
where the asymptotic variance
\begin{equation}
\label{SigD}
\sigma^2_{D} = 4 \sigma^2_{\rho}.
\end{equation}
\end{lem}
\begin{proof}
The proofs of Lemma \ref{LEM_CVG_TLC_THETA}, Lemma \ref{LEM_CVG_TLC_RHO} and Lemma \ref{LEM_CVG_TLC_D} may be found in \cite{BercuProia11}.
\end{proof}

\medskip

Our objective is to establish a set of moderate deviation principles on these estimates in order to get a better asymptotic precision than the central limit theorem. In all the sequel, $(b_{n})$ will denote a sequence of increasing positive numbers satisfying $1 = o(b_{n}^2)$ and $b_{n}^2 = o(n)$, that is
\begin{equation}
\label{BN_RATE}
b_n \longrightarrow \infty, \hspace{1.5cm} \frac{b_n}{\sqrt n} \longrightarrow 0.
\end{equation}

\medskip

\noindent{\bf Remarks and Notations.} \textit{In the whole paper, for any matrix $M$, $M^{\prime}$ and $\Vert M \Vert$ stand for the transpose and the euclidean norm of $M$, respectively. For any square matrix $M$, $\det(M)$ and $\rho(M)$ are the determinant and the spectral radius of $M$, respectively. Moreover, we will shorten \textnormal{large deviation principle} by LDP. In addition, for a sequence of random variables $(Z_{n})_{n}$ on $\dR^{d \times p}$, we say that $(Z_{n})_{n}$ converges $(b_{n}^2)-$superexponentially fast in probability to some random variable $Z$ if, for all $\delta > 0$,
\begin{equation*}
\limsup_{n \rightarrow \infty} \frac{1}{b_{n}^2} \log \dP\Big( \left\Vert Z_{n} - Z \right\Vert > \delta \Big) = -\infty.
\end{equation*}
This \textnormal{exponential convergence} with speed $b_{n}^2$ will be shortened as
\begin{equation*}
Z_{n} \superexp Z.
\end{equation*}
The \textnormal{exponential equivalence} with speed $b_{n}^2$ between two sequences of  random variables $(Y_{n})_{n}$ and $(Z_{n})_{n}$, whose precise definition is given in Definition 4.2.10 of \cite{DemboZeitouni98}, will be shortened as
\begin{equation*}
Y_{n} \superexpequiv Z_{n}.
\end{equation*}}

\section{On moderate deviations under the Gaussian condition}

In this first part, we focus our attention on moderate deviations for the Durbin-Watson statistic in the easy case where the driven noise $(V_{n})$ is normally distributed. This restrictive assumption allows us to reduce the set of hypothesis to the existence of $t > 0$ such that

\begin{enumerate}[\bf{({G}.}1)]
\item \label{HYP_G1}
\begin{equation*}
\dE\Big[ \exp(t \veps_0^2) \Big] < \infty,
\end{equation*}
\item \label{HYP_G2}
\begin{equation*}
\dE\Big[ \exp(t X_0^2) \Big] < \infty.
\end{equation*}
\end{enumerate}

\begin{thm}
\label{THM_PDMGAUSS_THETA}
Assume that there exists $t > 0$ such that \G{HYP_G1} and \G{HYP_G2} are satisfied. Then, the sequence
\begin{equation*}
\left( \frac{\sqrt{n}}{b_{n}} \left( \tn - \theta^{*} \right) \right)_{n \geq 1}
\end{equation*}
satisfies an LDP on $\dR$ with speed $b_{n}^2$ and good rate function
\begin{equation}
\label{Theta_Gauss_Rate}
I_{\theta}(x) = \frac{x^2}{2 \sigma^2_{\theta}}
\end{equation}
where $\sigma^2_{\theta}$ is given by \eqref{SigT}.
\end{thm}

\begin{thm}
\label{THM_PDMGAUSS_RHO}
Assume that there exists $t > 0$ such that \G{HYP_G1} and \G{HYP_G2} are satisfied. Then, as soon as $\theta \neq -\rho$, the sequence
\begin{equation*}
\left( \frac{\sqrt{n}}{b_{n}} \begin{pmatrix}
\tn - \theta^{*} \\
\rn - \rho^{*}
\end{pmatrix} \right)_{n \geq 1}
\end{equation*}
satisfies an LDP on $\dR^2$ with speed $b_{n}^2$ and good rate function
\begin{equation}
\label{Joint_Gauss_Rate}
K(x) = \frac{1}{2} x^{\prime} \Gamma^{-1} x
\end{equation}
where $\Gamma$ is given by \eqref{Gamma}. In particular, the sequence
\begin{equation*}
\left( \frac{\sqrt{n}}{b_{n}} \Big( \rn - \rho^{*} \Big) \right)_{n \geq 1}
\end{equation*}
satisfies an LDP on $\dR$ with speed $b_{n}^2$ and good rate function
\begin{equation}
\label{Rho_Gauss_Rate}
I_{\rho}(x) = \frac{x^2}{2 \sigma^2_{\rho}}
\end{equation}
where $\sigma^2_{\rho}$ is given by \eqref{SigR}.
\end{thm}

\begin{rem}
\label{REM_INV_GAMMA}
The covariance matrix $\Gamma$ is invertible if and only if $\theta\neq - \rho$ since one can see by a straightforward calculation that
\begin{equation*}
\det(\Gamma)=  \frac{\sigma^2_\theta (\theta+\rho)^2(1-\theta \rho)}{(1+\rho^2)}.
\end{equation*}
Moreover, in the particular case where $\theta= -\rho$, the sequences
\begin{equation*}
\left( \frac{\sqrt{n}}{b_{n}} \left( \tn - \theta^{*} \right) \right)_{n \geq 1} \hspace{1cm} \text{and} \hspace{1cm} \left( \frac{\sqrt{n}}{b_{n}} \Big( \rn - \rho^{*} \Big) \right)_{n \geq 1}
\end{equation*}
satisfy LDP on $\dR$ with speed $b_{n}^2$ and good rate functions respectively given by
\begin{equation*}
I_{\theta}(x) = \frac{x^2 (1 - \theta^2)}{2 (1 + \theta^2)} \hspace{1cm} \text{and} \hspace{1cm} I_{\rho}(x) = \frac{x^2 (1 - \theta^2)}{2 \theta^4 (1 + \theta^2)}.
\end{equation*}
\end{rem}

\begin{thm}
\label{THM_PDMGAUSS_D}
Assume that there exists $t > 0$ such that \G{HYP_G1} and \G{HYP_G2} are satisfied. Then, the sequence
\begin{equation*}
\left( \frac{\sqrt{n}}{b_{n}} \left( \dn - D^{*} \right) \right)_{n \geq 1}
\end{equation*}
satisfies an LDP on $\dR$ with speed $b_{n}^2$ and good rate function
\begin{equation}
\label{D_Gauss_Rate}
I_{D}(x) = \frac{x^2}{2 \sigma^2_{D}}
\end{equation}
where $\sigma^2_{D}$ is given by \eqref{SigD}.
\end{thm}

\begin{proof}
Theorem \ref{THM_PDMGAUSS_THETA}, Theorem \ref{THM_PDMGAUSS_RHO} and Theorem \ref{THM_PDMGAUSS_D} are proved in Section 4.
\end{proof}

\section{On moderate deviations under the Chen-Ledoux type condition}

Via an extensive use of Puhalskii's result, we will now focus our attention on the more general framework where the driven noise $(V_{n})$ is assumed to satisfy the Chen-Ledoux type condition. Accordingly, one shall introduce the following hypothesis, for $a = 2$ and $a = 4$.

\begin{enumerate}[\bf{({CL}.}1)]
\item \label{HYP_CL} Chen-Ledoux.
\begin{equation*}
\limsup_{n \rightarrow \infty} \frac{1}{b_{n}^2} \log n \dP\Big( \vert V_1 \vert^{a} > b_{n} \sqrt{n} \Big) = -\infty.
\end{equation*}
\item \label{HYP_E0}
\begin{equation*}
\frac{\vert \veps_0 \vert^{a}}{b_{n} \sqrt{n}} \superexp 0.
\end{equation*}
\item \label{HYP_X0}
\begin{equation*}
\frac{\vert X_0 \vert^{a}}{b_{n} \sqrt{n}} \superexp 0.
\end{equation*}
\end{enumerate}

\begin{rem}
If the random variable $V_1$ satisfies \CL{HYP_CL} with $a = 2$, then
\begin{equation}
\label{Cvg_Chen_Ledoux}
\limsup_{n \rightarrow \infty} \frac{1}{b_{n}^2} \log n \dP\Big( \left\vert V_1^2 - \dE[V_1^2] \right\vert > b_{n} \sqrt{n} \Big) = -\infty,
\end{equation}
which implies in particular that $\textnormal{Var}(V_1^4) < \infty$. Moreover, if the random variable $V_1$ has exponential moments, i.e. if there exists $t > 0$ such that
\begin{equation*}
\dE\Big[ \exp{(t V_1^2)} \Big] < \infty,
\end{equation*}
then \CL{HYP_CL} is satisfied for every increasing sequence $(b_{n})$. From \cite{Arcones01}, \cite{EichelsbacherLowe03}, condition \eqref{Cvg_Chen_Ledoux} is equivalent to say that the sequence
\begin{equation*}
\left( \frac{1}{b_{n} \sqrt{n}} \sum_{k=1}^{n} \Big( V_{k}^2 - \dE[V_{k}^2] \Big) \right)_{n \geq 1}
\end{equation*}
satisfies an LDP on $\dR$ with speed $b_{n}^2$ and good rate function
\begin{equation*}
I(x) = \frac{x^2}{2 \textnormal{Var}(V_1^2)}.
\end{equation*}
\end{rem}

\begin{rem} If we choose $b_n=n^{\alpha}$ with $0<\alpha<1/2$, \CL{HYP_CL} is immediately satisfied if there exists $t > 0$ and $0 < \beta < 1$ such that
\begin{equation*}
\dE\Big[ \exp{(t V_1^{2 \beta})} \Big] < \infty,
\end{equation*}
which is clearly a weaker assumption than the existence of $t > 0$ such that
\begin{equation*}
\dE\Big[ \exp{(t V_1^2)} \Big] < \infty,
\end{equation*}
imposed in the previous section.
\end{rem}

\begin{rem}
If \CL{HYP_CL} is satisfied for $a=4$, then it is also satisfied for all $0<b<a$.
\end{rem}

\begin{rem}
In the technical proofs that will follow, rather than \CL{HYP_CL} with $a=4$, the weakest assumption really needed could be summarized by the existence of a
large constant $C$ such that
\begin{equation*}
\limsup_{n \rightarrow \infty} \frac{1}{b_n^2}\log\dP\left(\frac{1}{n} \sum_{k=1}^{n} V_{k}^4 > C \right)=-\infty.
\end{equation*}
\end{rem}

\begin{thm}
\label{THM_PDMCL_THETA}
Assume that \CL{HYP_CL}, \CL{HYP_E0} and \CL{HYP_X0} are satisfied. Then, the sequence
\begin{equation*}
\left( \frac{\sqrt{n}}{b_{n}} \left( \tn - \theta^{*} \right) \right)_{n \geq 1}
\end{equation*}
satisfies the LDP on $\dR$ given in Theorem \ref{THM_PDMGAUSS_THETA}.
\end{thm}

\begin{thm}
\label{THM_PDMCL_RHO}
Assume that \CL{HYP_CL}, \CL{HYP_E0} and \CL{HYP_X0} are satisfied. Then, as soon as $\theta \neq -\rho$, the sequence
\begin{equation*}
\left( \frac{\sqrt{n}}{b_{n}} \begin{pmatrix}
\tn - \theta^{*} \\
\rn - \rho^{*}
\end{pmatrix} \right)_{n \geq 1}
\end{equation*}
satisfies the LDP on $\dR^2$ given in Theorem \ref{THM_PDMGAUSS_RHO}. In particular, the sequence
\begin{equation*}
\left( \frac{\sqrt{n}}{b_{n}} \Big( \rn - \rho^{*} \Big) \right)_{n \geq 1}
\end{equation*}
satisfies the LDP on $\dR$ also given in Theorem \ref{THM_PDMGAUSS_RHO}.
\end{thm}

\begin{rem}
We have already seen in Remark \ref{REM_INV_GAMMA} that the covariance matrix $\Gamma$ is invertible if and only if $\theta\neq - \rho$. In the particular case where $\theta= -\rho$, the sequences
\begin{equation*}
\left( \frac{\sqrt{n}}{b_{n}} \left( \tn - \theta^{*} \right) \right)_{n \geq 1} \hspace{1cm} \text{and} \hspace{1cm} \left( \frac{\sqrt{n}}{b_{n}} \Big( \rn - \rho^{*} \Big) \right)_{n \geq 1}
\end{equation*}
satisfy the LDP on $\dR$ given in Remark \ref{REM_INV_GAMMA}.
\end{rem}

\begin{thm}
\label{THM_PDMCL_D}
Assume that \CL{HYP_CL}, \CL{HYP_E0} and \CL{HYP_X0} are satisfied. Then, the sequence
\begin{equation*}
\left( \frac{\sqrt{n}}{b_{n}} \left( \dn - D^{*} \right) \right)_{n \geq 1}
\end{equation*}
satisfies the LDP on $\dR$ given in Theorem \ref{THM_PDMGAUSS_D}.
\end{thm}

\begin{proof}
Theorem \ref{THM_PDMCL_THETA}, Theorem \ref{THM_PDMCL_RHO} and Theorem \ref{THM_PDMCL_D} are proved in Section 4.
\end{proof}

\section{Proof of the main results}
For a matter of readability, some notations commonly used in the following proofs have to be introduced. First, for all $n \geq 1$, let
\begin{equation}
\label{L} L_{n} = \sum_{k=1}^{n} V_{k}^2.
\end{equation}
Then, let us define $M_{n}$, for all $n \geq 1$, as
\begin{equation}
\label{M} M_{n} = \sum_{k=1}^{n} X_{k-1} V_{k}
\end{equation}
where $M_0 = 0$. For all $n \geq 1$, denote by $\cF_n$ the $\sigma$-algebra of the events occurring up to time $n$,  $\cF_n = \sigma(X_0,\veps_0,V_1,\cdots,V_n)$. We infer from $(\ref{M})$ that $(M_n)_{n\ge 0}$ is a locally square-integrable real martingale with respect to the filtration $\dF = (\cF_n)_{n \geq 0}$ with predictable quadratic variation given by $\langle M \rangle_0=0$ and for all $n \geq 1$, $\langle M \rangle_{n} = \sigma^2 S_{n-1}$, where
\begin{equation}
\label{S} S_{n} =\sum_{k=0}^{n} X_{k}^2.
\end{equation}
Moreover, $(N_n)_{n\ge 0}$ is defined, for all $n \geq 2$, as
\begin{equation}
\label{N} N_{n} =\sum_{k=2}^{n} X_{k-2} V_{k}
\end{equation}
and $N_0 = N_1 = 0$. It is not hard to see that $(N_n)_{n \geq 0}$ is also a locally square-integrable real martingale sharing the same properties than  $(M_n)_{n \geq 0}$. More precisely, its predictable quadratic variation is given by $\langle N \rangle_{n} = \sigma^2 S_{n-2}$. To conclude, let $P_0 = 0$ and, for all $n \geq 1$,
\begin{equation}
\label{P} P_{n} =\sum_{k=1}^{n} X_{k-1} X_{k}.
\end{equation}

\subsection*{}
\begin{center}
{\bf 4.1. Proof of Theorem \ref{THM_PDMGAUSS_THETA}.}
\end{center}
Before starting the proof of Theorem \ref{THM_PDMGAUSS_THETA}, we need to introduce some technical tools. Denote by $\ell$ the almost sure limit of $S_{n}/n$ \cite{BercuProia11}, given by
\begin{equation}
\label{LimS}
\ell = \frac{\sigma^2 (1 + \theta\rho)}{(1-\theta^2)(1 - \theta\rho)(1 - \rho^2)}.
\end{equation}

\begin{lem}
\label{LEM_CVGEXP_SN}
Under the assumptions of Theorem \ref{THM_PDMGAUSS_THETA}, we have the exponential convergence
\begin{equation}
\label{LIMEXP_SN}
\frac{S_{n}}{n} \superexp \ell
\end{equation}
where $\ell$ is given by \eqref{LimS}.
\end{lem}
\begin{proof}
After straightforward calculations, we get that for all $n\geq 2$,
\begin{equation}
\label{Decomp_Sn}
\frac{S_n}{n} - \ell = \frac{\ell}{\sigma^2} \Bigg[ \left(\frac{L_n}{n} - \sigma^2 \right) + 2 \theta^{*} \frac{M_n}{n} - 2\theta\rho\frac{N_n}{n} + \frac{R_n}{n} \Bigg]
\end{equation}
where $L_n$, $M_n$, $S_n$ and $N_n$ are respectively given by \eqref{L},
\eqref{M}, \eqref{S} and \eqref{N},
\begin{equation*}
R_n = [2(\theta + \rho) \rho^{*} - (\theta+\rho)^2 - (\theta\rho)^2]X_n^2 - (\theta\rho)^2 X_{n-1}^2 + 2 \rho^{*} X_n X_{n-1} + \xi_1,
\end{equation*}
and where the remainder term
\begin{equation*}
\xi_1 = (1-2\theta\rho-\rho^2) X_0^2 + \rho^2\veps_0^2 + 2\theta\rho
X_0\veps_0 - 2\rho \rho^{*} (\veps_0-X_0)X_0 + 2\rho(\veps_0-X_0)V_1.
\end{equation*}

\medskip

First of all, $(V_{n})$ is a sequence of independent and identically distributed gaussian random variables with zero mean and variance $\sigma^2 > 0$. It immediately follows from Cram\'er-Chernoff's Theorem, expounded e.g. in \cite{DemboZeitouni98}, that for all $\delta > 0$,
\begin{equation}
\label{Pgd_Ln_Gauss}
\limsup_{n \rightarrow \infty} \frac{1}{n} \log \dP\left( \left\vert \frac{L_{n}}{n} - \sigma^2 \right\vert > \delta \right) < 0.
\end{equation}
Since $b_{n}^2 = o(n)$, the latter convergence leads to
\begin{equation}
\label{Pdm_Ln_Gauss}
 \frac{L_{n}}{n} \superexp \sigma^2,
\end{equation}
ensuring the exponential convergence of $L_{n}/n$ to $\sigma^2$ with speed $b_{n}^2$. Moreover, for all $\delta > 0$ and a suitable $t > 0$, we clearly obtain from Markov's inequality that
\begin{equation*}
\dP\left( \frac{X_0^2}{n} > \delta \right) \leq \exp\left( -t n \delta \right) \dE \Big[ \exp(t X_0^2) \Big],
\end{equation*}
which immediately implies via \G{HYP_G2},
\begin{equation}
\label{Pdm_X0_Gauss}
 \frac{X_0^2}{n} \superexp 0,
\end{equation}
and we get the exponential convergence of $X_0^2/n$ to 0 with speed $b_{n}^2$. The same is true for $V_1^2/n$, $\veps_0^2/n$ and more generally for any isolated term of order 2 in relation \eqref{Decomp_Sn} whose numerator do not depend on $n$. Let us now focus our attention on $X_{n}^2/n$. The model \eqref{AR} can be rewritten in the vectorial form,
\begin{equation}
\label{Mod_Vect}
\Phi_{n} = A \Phi_{n-1} + W_{n}
\end{equation}
where $\Phi_{n} = \begin{pmatrix} X_{n} & X_{n-1} \end{pmatrix}^{\prime}$ stands for the lag vector of order 2, $W_{n} = \begin{pmatrix} V_{n} & 0 \end{pmatrix}^{\prime}$ and
\begin{equation}
\label{Mod_Vect_A}
A = \begin{pmatrix}
\theta+\rho & -\theta\rho \\
1 & 0
\end{pmatrix}.
\end{equation}
It is easy to show that $\rho(A) = \text{max}(\vert \theta \vert, \vert \rho \vert) < 1$ under the stability conditions. According to Proposition 4.1 of \cite{Worms00},
\begin{equation*}
\frac{\Vert \Phi_{n} \Vert^2}{n} \superexpn 0,
\end{equation*}
which is clearly sufficient to deduce that
\begin{equation}
\label{Pdm_Xn_Gauss}
 \frac{X_n^2}{n} \superexp 0.
\end{equation}
The exponential convergence of $R_{n}/n$ to 0 with speed $b_{n}^2$ is achieved following exactly the same lines. To conclude the proof of Lemma \ref{LEM_CVGEXP_SN}, it remains to study the exponential asymptotic behavior of $M_{n}/n$. For all $\delta > 0$ and a suitable $y > 0$,
\begin{eqnarray}
\label{Ineg_GaussMart}
\dP\left( \frac{M_{n}}{n} > \delta \right) & = & \dP\left( \frac{M_{n}}{n} > \delta, \langle M \rangle_{n} \leq y \right) + \dP\left( \frac{M_{n}}{n} > \delta, \langle M \rangle_{n} > y \right), \nonumber \\
 & \leq & \exp\left( -\frac{n^2 \delta^2}{2 y} \right) + \dP\Big( \langle M \rangle_{n} > y \Big),
\end{eqnarray}
by application of Theorem 4.1 of \cite{BercuTouati08} in the case of a gaussian martingale. Then, noting that we have the following inequality,
\begin{equation}
\label{Ineg_Sn}
S_{n} \leq \alpha X_0^2 + \beta \veps_0^2 + \beta L_{n} \hspace{0.5cm} \textnormal{a.s.}
\end{equation}
with $\alpha = 1 + \left( 1 - \vert \theta \vert \right)^{-2}$ and $\beta = \left( 1 - \vert \rho \vert \right)^{-2} \left( 1 - \vert \theta \vert \right)^{-2}$, we get for a suitable $t >0$,
\begin{eqnarray*}
\dP\Big( \langle M \rangle_{n} > y \Big) & \leq & \dP\bigg( X_0^2 > \frac{y}{3 \alpha \sigma^2} \bigg) + \dP\left( \veps_0^2 > \frac{y}{3 \beta \sigma^2} \right) + \dP\left( L_{n-1} > \frac{y}{3 \beta \sigma^2} \right), \\
 & \leq & \exp\left( \frac{- y t}{3 \alpha \sigma^2} \right) \dE \Big[ \exp(t X_0^2) \Big] + \exp\left( \frac{- y t}{3 \beta \sigma^2} \right) \dE \Big[ \exp(t \veps_0^2) \Big] \\
 & & \hspace{1cm} + \hspace{0.1cm} \dP\left( L_{n-1} > \frac{y}{3 \beta \sigma^2} \right), \\
 & \leq & 3 \max\bigg( \exp\left( \frac{- y t}{3 \alpha \sigma^2} \right) \dE \Big[ \exp(t X_0^2) \Big], \exp\left( \frac{- y t}{3 \beta \sigma^2} \right) \dE \Big[ \exp(t \veps_0^2) \Big], \\
& & \hspace{1cm} \dP\left( L_{n-1} > \frac{y}{3 \beta \sigma^2} \right) \bigg).
\end{eqnarray*}
Let us choose $y = n x$, assuming $x > 3 \beta \sigma^4$. It follows that
\begin{eqnarray*}
\frac{1}{b_{n}^2} \log \dP\Big( \langle M \rangle_{n} > n x \Big) & \leq & \frac{\log 3}{b_{n}^2} +  \frac{1}{b_{n}^2} \max\bigg( \frac{-n x t}{3 \alpha \sigma^2} + \log \dE \Big[ \exp(t X_0^2) \Big], \\
 & & \hspace{1cm} \frac{-n x t}{3 \beta \sigma^2} + \log \dE \Big[ \exp(t \veps_0^2) \Big], \log \dP\left( L_{n-1} > \frac{n x}{3 \beta \sigma^2} \right) \bigg).
\end{eqnarray*}
Since $b_{n}^2 = o(n)$ and by virtue of \eqref{Pdm_Ln_Gauss} with $\delta = x / (3 \beta \sigma^2) - \sigma^2 > 0$, we obtain that
\begin{equation}
\label{Cvg_Exp_Proc_Gauss}
\limsup_{n \rightarrow \infty} \frac{1}{b_{n}^2} \log \dP\Big( \langle M \rangle_{n} > n x \Big) = -\infty.
\end{equation}
It enables us by \eqref{Ineg_GaussMart} to deduce that for all $\delta > 0$,
\begin{equation}
\label{Pdm_Mn_Gauss}
\limsup_{n \rightarrow \infty} \frac{1}{b_{n}^2} \log \dP\left( \frac{M_{n}}{n} > \delta \right) = -\infty.
\end{equation}
The same result is also true replacing $M_{n}$ by $-M_{n}$ in \eqref{Pdm_Mn_Gauss} since $M_{n}$ and $-M_{n}$ share the same distribution. Therefore, we find that
\begin{equation}
\label{LIMEXP_MN}
\frac{M_{n}}{n} \superexp 0.
\end{equation}
A similar reasoning leads to the exponential convergence of $N_{n}/n$ to 0, with speed $b_{n}^2$. Finally, we obtain \eqref{LIMEXP_SN} from \eqref{Decomp_Sn}  together with \eqref{Pdm_Ln_Gauss}, \eqref{Pdm_X0_Gauss}, \eqref{Pdm_Xn_Gauss} and \eqref{LIMEXP_MN} which achieves the proof of Lemma \ref{LEM_CVGEXP_SN}.
\end{proof}

\begin{cor}
\label{COR_CVGEXP_PN}
By virtue of Lemma \ref{LEM_CVGEXP_SN} and under the same assumptions, we have the exponential convergence
\begin{equation}
\label{LIMEXP_PN}
\frac{P_{n}}{n} \superexp \ell_1
\end{equation}
where $\ell_1 = \theta^{*} \ell$.
\end{cor}
\begin{proof}
The proof of Corollary \ref{COR_CVGEXP_PN} is immediately derived from the following inequality,
\begin{eqnarray}
\label{Decomp_Pn}
\left\vert \frac{P_n}{n} - \theta^*\frac{S_n}{n} \right\vert & = & \left\vert \frac{1}{1+\theta\rho}\frac{M_n}{n} + \frac{1}{1+\theta\rho}\frac{R_{n}(\theta)}{n} - \theta^*\frac{X_n^2}{n}  \right\vert, \nonumber\\
& \leq & \frac{1}{1+\theta\rho}\frac{\vert M_n \vert}{n} + \frac{1}{1+\theta\rho} \frac{\vert R_{n}(\theta) \vert}{n} + \vert \theta^* \vert \frac{X_n^2}{n}
\end{eqnarray}
with $R_{n}(\theta) = \theta\rho X_{n} X_{n-1} + \rho X_0(\veps_0 - X_0)$.
\end{proof}
We are now in the position to prove Theorem \ref{THM_PDMGAUSS_THETA}. We shall make use of the following deviation principle for martingales established by Worms \cite{Worms99}.
\begin{thm}[Worms]
\label{THM_WORMS_GAUSS}
Let $(Y_{n})$ be an adapted sequence with values in $\dR^p$, and $(V_{n})$ a gaussian noise with variance $\sigma^2 > 0$. We suppose that $(Y_{n})$ satisfies, for some invertible square matrix $C$ of order $p$ and a speed sequence $(b_{n}^2)$ such that $b_{n}^2 = o(n)$, the exponential convergence for any $\delta > 0$,
\begin{equation}
\label{Pgd_Sn_Gauss}
\lim_{n \rightarrow \infty} \frac{1}{b_{n}^2} \log \dP\left( \left\Vert \frac{1}{n} \sum_{k=0}^{n-1} Y_{k} Y_{k}^{\prime} - C \right\Vert > \delta \right) = -\infty.
\end{equation}
Then, the sequence
\begin{equation*}
\left( \frac{M_{n}}{b_{n}\sqrt{n}} \right)_{n \geq 1}
\end{equation*}
satisfies an LDP on $\dR^{p}$ of speed $b_{n}^2$ and good rate function
\begin{equation}
\label{Pgd_Worms_Rate_Gauss}
I(x) = \frac{1}{2 \sigma^2} x^{\prime} C^{-1} x
\end{equation}
where $(M_{n})$ is the martingale given by
\begin{equation*}
M_{n} = \sum_{k=1}^{n} Y_{k-1} V_{k}.
\end{equation*}
\end{thm}
\begin{proof}
The proof of Theorem \ref{THM_WORMS_GAUSS} is contained in the one of Theorem 5 of \cite{Worms99} with $d=1$.
\end{proof}

\medskip

\noindent{\bf Proof of Theorem \ref{THM_PDMGAUSS_THETA}.} Let us consider the decomposition
\begin{equation}
\label{Decomp_Theta}
\frac{\sqrt{n}}{b_{n}} \left( \tn - \theta^{*} \right) = \frac{\sqrt{n}}{b_{n}} \left( \frac{\sigma^2}{1 + \theta\rho} \right) \frac{M_{n}}{\langle M \rangle_{n}} + \frac{\sqrt{n}}{b_{n}} \left( \frac{1}{1 + \theta\rho} \right) \frac{R_{n}(\theta)}{S_{n-1}},
\end{equation}
that can be obtained by a straighforward calculation, where the remainder term $R_{n}(\theta)$ is defined in \eqref{Decomp_Pn}. First, by using the same methodology as in convergence \eqref{Pdm_X0_Gauss}, we obtain that for all $\delta > 0$ and for a suitable $t > 0$,
\begin{eqnarray}
\label{Cvg_Pdm_X0}
\limsup_{n \rightarrow \infty} \frac{1}{b_{n}^2} \log \dP\left( \frac{X_0^2}{b_{n} \sqrt{n}} > \delta \right) & \leq & \lim_{n \rightarrow \infty} \left( -t \delta \frac{\sqrt{n}}{b_{n}} \right) + \lim_{n \rightarrow \infty} \frac{1}{b_{n}^2} \log \dE \Big[ \exp(t X_0^2) \Big], \nonumber \\
 & = & -\infty,
\end{eqnarray}
since $b_{n} = o(\sqrt{n})$, and the same goes for any isolated term in \eqref{Decomp_Theta} of order 2 whose numerator do not depend on $n$. Moreover, under the gaussian assumption on the driven noise $(V_{n})$, it is not hard to see that
\begin{equation}
\label{Cvg_Pdm_MaxVn}
\frac{1}{b_{n} \sqrt{n}} \max_{1 \leq k \leq n} V_{k}^2 \hspace{0.1cm} \superexp 0.
\end{equation}
As a matter of fact, for all $\delta > 0$ and for all $t > 0$,
\begin{eqnarray*}
\dP\left( \max_{1 \leq k \leq n} V_{k}^2 \geq \delta b_{n} \sqrt{n} \right) & = & \dP\left( \bigcup_{k=1}^{n} \Big\{ V_{k}^2 \geq \delta b_{n} \sqrt{n} \Big\} \right) \hspace{0.2cm} \leq \hspace{0.3cm} \sum_{k=1}^{n} \dP\Big( V_{k}^2 \geq \delta b_{n} \sqrt{n} \Big), \\
& \leq & n \exp\left( - t \delta b_{n} \sqrt{n} \right) \dE\Big[ \exp\left( t V_1^2 \right) \Big].
\end{eqnarray*}
In addition, as soon as $0 < t < 1/(2 \sigma^2)$, $\dE\big[ \exp( t V_1^2 ) \big] < \infty$. Consequently,
\begin{eqnarray*}
\frac{1}{b_{n}^2} \log \dP\left( \max_{1 \leq k \leq n} V_{k}^2 \geq \delta b_{n} \sqrt{n} \right) & \leq & \frac{\log n}{b_{n}^2} - \frac{t \delta \sqrt{n}}{b_{n}} + \frac{\log \dE\Big[ \exp\left( t V_1^2 \right) \Big]}{b_{n}^2}, \\
& \leq & \frac{\sqrt{n}}{b_{n}} \left( \frac{\log n}{b_{n} \sqrt{n}} - t \delta + \frac{\log \dE\Big[ \exp\left( t V_1^2 \right) \Big]}{b_{n} \sqrt{n}} \right)
\end{eqnarray*}
which clearly leads to \eqref{Cvg_Pdm_MaxVn}. Furthermore, it follows from \eqref{AR} that
\begin{equation}
\label{Ineg_MaxXn}
\max_{1 \leq k \leq n} X_{k}^2 \leq \frac{1}{1 - \vert \theta \vert} X_0^2 + \left( \frac{1}{1 - \vert \theta \vert} \right)^{2} \max_{1 \leq k \leq n} \veps_{k}^2,
\end{equation}
as well as
\begin{equation}
\label{Ineg_MaxEn}
\max_{1 \leq k \leq n} \veps_{k}^2 \leq \frac{1}{1 - \vert \rho \vert} \veps_0^2 + \left( \frac{1}{1 - \vert \rho \vert} \right)^{2} \max_{1 \leq k \leq n} V_{k}^2.
\end{equation}
Then, we deduce from \eqref{Cvg_Pdm_X0}, \eqref{Cvg_Pdm_MaxVn}, \eqref{Ineg_MaxXn} and \eqref{Ineg_MaxEn} that
\begin{equation*}
\frac{1}{b_{n} \sqrt{n}} \max_{1 \leq k \leq n} \veps_{k}^2 \hspace{0.1cm} \superexp 0 \hspace{1cm} \text{and} \hspace{1cm} \frac{1}{b_{n} \sqrt{n}} \max_{1 \leq k \leq n} X_{k}^2 \hspace{0.1cm} \superexp 0,
\end{equation*}
which of course imply the exponential convergence of $X_{n}^2/(b_{n}\sqrt{n})$ to 0, with speed $b_{n}^2$. Therefore, we obtain that
\begin{equation}
\label{Cvg_Pdm_Rn}
\frac{R_n(\theta)}{b_{n} \sqrt{n}} \superexp 0.
\end{equation}
We infer from Lemma \ref{LEM_CVGEXP_SN} together with Lemma 4.1 of \cite{Worms00} that the following convergence is satisfied,
\begin{equation}
\label{Cvg_Inv_Sn}
\frac{n}{S_{n}} \superexp \frac{1}{\ell}
\end{equation}
where $\ell > 0$ is given by \eqref{LimS}. According to \eqref{Cvg_Pdm_Rn}, the latter convergence and again Lemma 4.1 of \cite{Worms00}, we deduce that
\begin{equation}
\label{Cvg_Pdm_Reste}
\frac{\sqrt{n}}{b_{n}} \left( \frac{1}{1 + \theta\rho} \right) \frac{R_{n}(\theta)}{S_{n-1}} \superexp 0.
\end{equation}
Hence, we obtain from \eqref{Cvg_Inv_Sn} that the same is true for
\begin{equation}
\label{Cvg_DiffS}
\frac{\sigma^2}{1 + \theta\rho} \frac{M_{n}}{b_{n} \sqrt{n}} \left( \frac{n}{\langle M \rangle_{n}} - \frac{1}{\sigma^2 \ell} \right) \superexp 0,
\end{equation}
since Lemma \ref{LEM_CVGEXP_SN} together with Theorem \ref{THM_WORMS_GAUSS} with $p=1$ directly show that $(M_{n} / (b_{n} \sqrt{n}))$ satisfies an LDP with speed $b_{n}^2$ and good rate function given, for all $x \in \dR$, by
\begin{equation}
\label{Pdm_Mn}
J(x) = \frac{x^2}{2 \ell \sigma^2}.
\end{equation}
As a consequence,
\begin{equation}
\label{Theta_Pdm_Equiv}
\frac{\sqrt{n}}{b_{n}} \left( \tn - \theta^{*} \right) \superexpequiv \frac{1}{\ell (1 + \theta\rho)} \frac{M_{n}}{b_{n} \sqrt{n}},
\end{equation}
and this implies that both of them share the same LDP, see e.g. \cite{DemboZeitouni98}. One shall now take advantage of the contraction principle \cite{DemboZeitouni98} to establish that $(\sqrt{n} (\tn - \theta^{*}) / b_{n})$ satisfies an LDP with speed $b_{n}^2$ and good rate function $I_{\theta}(x)$ given by \eqref{Theta_Gauss_Rate}. The contraction principle enables us to conclude that the good rate function of the LDP with speed $b_{n}^2$ associated with equivalence \eqref{Theta_Pdm_Equiv} is given by $I_{\theta}(x) = J(\ell (1+\theta\rho) x)$, that is
\begin{equation*}
I_{\theta}(x) = \frac{x^2}{2 \sigma^2_{\theta}},
\end{equation*}
which achieves the proof of Theorem \ref{THM_PDMGAUSS_THETA}. \hfill
$\mathbin{\vbox{\hrule\hbox{\vrule height1ex \kern.5em\vrule height1ex}\hrule}}$

\subsection*{}
\begin{center}
{\bf 4.2. Proof of Theorem \ref{THM_PDMGAUSS_RHO}.}
\end{center}
We need to introduce some more notations. For all $n \geq 2$, let
\begin{equation}
\label{Q} Q_{n} = \sum_{k=2}^{n} X_{k-2} V_{k}.
\end{equation}
In addition, for all $n \geq 1$, denote
\begin{equation}
\label{T}
T_{n} = 1+\theta^*\rho^* - \left( 1+\rho^*(\tn+\theta^*) \right) \frac{S_n}{S_{n-1}} + \left( 2\rho^*+\tn+\theta^* \right) \frac{P_n}{S_{n-1}} - \frac{Q_n}{S_{n-1}},
\end{equation}
where $S_{n}$ and $P_{n}$ are respectively given by \eqref{S} and \eqref{P}. Finally, for all $n \geq 0$, let
\begin{equation}
\label{J}
J_{n} = \sum_{k=0}^{n} \ek^{\hspace{0.1cm} 2}
\end{equation}
where the residual set $(\en)$ is given in \eqref{Est_Res}. A set of additional technical tools has to be expounded to make the proof of Theorem \ref{THM_PDMGAUSS_RHO} more tractable.

\begin{cor}
\label{COR_CVGEXP_QN}
By virtue of Lemma \ref{LEM_CVGEXP_SN} and under the same assumptions, we have the exponential convergence
\begin{equation*}
\frac{Q_{n}}{n} \superexp \ell_2
\end{equation*}
where $\ell_2 = ((\theta + \rho) \theta^{*} - \theta\rho) \ell$.
\end{cor}
\begin{proof}
The proof of Corollary \ref{COR_CVGEXP_QN} immediately follows from the inequality,
\begin{eqnarray}
\label{Decomp_Qn}
\left\vert \frac{Q_n}{n} - ((\theta + \rho) \theta^{*} - \theta\rho) \frac{S_n}{n} \right\vert & = & \left\vert \theta^{*} \frac{M_{n}}{n} + \frac{N_{n}}{n} + \frac{\xi_{n}^{Q}}{n} \right\vert, \nonumber\\
& \leq & \vert \theta^{*} \vert \frac{\vert M_{n} \vert}{n} + \frac{\vert N_{n} \vert}{n} + \frac{\vert \xi_{n}^{Q} \vert}{n}
\end{eqnarray}
where $\xi_{n}^{Q}$ is a residual made of isolated terms such that
\begin{equation*}
\frac{\xi_{n}^{Q}}{n} \superexp 0,
\end{equation*}
see e.g. the proof of Theorem 3.2 in \cite{BercuProia11} where more details are given on $\xi_{n}^{Q}$.
\end{proof}

\begin{lem}
\label{LEM_CVGEXP_AN}
Under the assumptions of Theorem \ref{THM_PDMGAUSS_RHO}, we have the exponential convergence
\begin{equation*}
A_{n} \superexp A
\end{equation*}
where
\begin{equation}
\label{A}
A_{n} = \frac{n}{1 + \theta\rho} \begin{pmatrix}
\displaystyle \frac{1}{S_{n-1}} & 0 \vspace{1ex} \\
\displaystyle \frac{T_{n}}{J_{n-1}} & \displaystyle -\frac{(\theta+\rho)}{J_{n-1}}
\end{pmatrix},
\end{equation}
and
\begin{equation}
\label{ALim}
A = \frac{1}{\ell (1 + \theta\rho) (1 - (\theta^{*})^2)} \begin{pmatrix}
1 - (\theta^{*})^2 & 0 \\
\theta\rho + (\theta^{*})^2 & -(\theta+\rho)
\end{pmatrix}.
\end{equation}
\end{lem}
\begin{proof}
Via \eqref{Cvg_Inv_Sn}, we directly obtain the exponential convergence,
\begin{equation}
\label{Cvg_A11}
\frac{1}{(1 + \theta\rho)} \frac{n}{S_{n-1}} \superexp \frac{1}{\ell (1 + \theta \rho)}.
\end{equation}
The combination of Lemma \ref{LEM_CVGEXP_SN}, Corollary \ref{COR_CVGEXP_PN}, Corollary \ref{COR_CVGEXP_QN} and Lemma 4.1 of \cite{Worms00} shows, after a simple calculation, that
\begin{equation}
\label{Cvg_T}
T_{n} \superexp (\theta^{*})^2 + \theta\rho.
\end{equation}
Moreover, $J_{n}$ given by \eqref{J} can be rewritten as
\begin{equation*}
J_{n} = S_{n} - 2 \tn P_{n} + \tn^{\hspace{0.1cm} 2} S_{n-1},
\end{equation*}
which leads, via Lemma 4.1 in \cite{Worms00}, to

\begin{equation}
\label{Cvg_J}
\frac{J_{n}}{n} \superexp \ell(1 - (\theta^{*})^2).
\end{equation}
Convergences \eqref{Cvg_T} and \eqref{Cvg_J} imply
\begin{equation}
\label{Cvg_A21}
\left( \frac{n}{1 + \theta\rho} \right) \frac{T_{n}}{J_{n-1}} \superexp \frac{(\theta^{*})^2 + \theta\rho}{\ell (1 + \theta \rho) (1 - (\theta^{*})^2)},
\end{equation}
and finally,
\begin{equation}
\label{Cvg_A22}
\left( \frac{n}{1 + \theta\rho} \right) \frac{\theta + \rho}{J_{n-1}} \superexp \frac{\theta + \rho}{\ell (1 + \theta\rho) (1 - (\theta^{*})^2)}.
\end{equation}
Finally, \eqref{Cvg_A11} together with \eqref{Cvg_A21} and \eqref{Cvg_A22} achieve the proof of Lemma \ref{LEM_CVGEXP_AN}.
\end{proof}

\medskip

\noindent{\bf Proof of Theorem \ref{THM_PDMGAUSS_RHO}.} We shall make use of the decomposition
\begin{equation}
\label{Decomp_Joint}
\frac{\sqrt{n}}{b_{n}} \begin{pmatrix}
\tn - \theta^{*} \\
\rn - \rho^{*}
\end{pmatrix} = \frac{1}{b_{n} \sqrt{n}} A_{n} Z_{n} + B_{n},
\end{equation}
where $A_{n}$ is given by \eqref{A}, $(Z_{n})_{n \geq 0}$ is the 2-dimensional vector martingale given by
\begin{equation}
\label{Z}
Z_{n} =  \begin{pmatrix}
M_{n} \\
N_{n}
\end{pmatrix},
\end{equation}
and where the remainder term
\begin{equation}
\label{B}
B_{n} = \frac{1}{(1 + \theta\rho)} \frac{\sqrt{n}}{b_{n}} \begin{pmatrix}
\displaystyle \frac{R_{n}(\theta)}{S_{n-1}} \vspace{1ex} \\
\displaystyle \frac{R_{n}(\rho)}{J_{n-1}}
\end{pmatrix}.
\end{equation}
The first component $R_{n}(\theta)$ is given in \eqref{Decomp_Pn} while $R_{n}(\rho)$, whose definition may be found in the proof of Theorem 3.2 in \cite{BercuProia11}, is made of isolated terms. Consequently, \eqref{Cvg_Pdm_X0} and \eqref{Cvg_Pdm_Rn} are sufficient to ensure that
\begin{equation*}
\frac{R_{n}(\theta)}{b_{n} \sqrt{n}} \superexp 0 \hspace{1cm} \text{and} \hspace{1cm} \frac{R_{n}(\rho)}{b_{n} \sqrt{n}} \superexp 0.
\end{equation*}
Therefore, we obtain that
\begin{equation}
\label{Cvg_B}
B_{n} \superexp 0.
\end{equation}
In addition, it follows from Lemma \ref{LEM_CVGEXP_AN} and Theorem \ref{THM_WORMS_GAUSS} with $p=2$ that $(Z_{n}/(b_{n}\sqrt{n}))$ satisfies an LDP on $\dR^2$ with speed $b_{n}^2$ and good rate function given, for all $x \in \dR^2$, by
\begin{equation}
\label{Pdm_Zn}
J(x) = \frac{1}{2 \sigma^2} x^{\prime} \Lambda^{-1} x,
\end{equation}
where
\begin{equation}
\label{Lam}
\Lambda = \ell \begin{pmatrix}
1 & \theta^{*} \\
\theta^{*} & 1
\end{pmatrix},
\end{equation}
since we have the exponential convergence
\begin{equation}
\label{Cvg_Exp_Proc_Joint_Gauss}
\frac{\langle Z \rangle_{n}}{n} \superexp \sigma^2 \Lambda
\end{equation}
by application of Lemma \ref{LEM_CVGEXP_SN} and Corollary \ref{COR_CVGEXP_PN}. One observes that $\det(\Lambda) = \ell^2 (1 - (\theta^{*})^2) > 0$ implying that $\Lambda$ is invertible. As a consequence,
\begin{equation}
\label{Cvg_DiffA}
\frac{1}{b_{n} \sqrt{n}} (A_{n} - A) Z_{n} \superexp 0,
\end{equation}
and we deduce from \eqref{Decomp_Joint} that
\begin{equation}
\label{Joint_Pdm_Equiv}
\frac{\sqrt{n}}{b_{n}} \begin{pmatrix}
\tn - \theta^{*} \\
\rn - \rho^{*} \\
\end{pmatrix}
\superexpequiv \frac{1}{b_{n} \sqrt{n}} A Z_{n}.
\end{equation}
This of course implies that both of them share the same LDP. The contraction principle \cite{DemboZeitouni98} enables us to conclude that the rate function of the LDP on $\dR^2$ with speed $b_{n}^2$ associated with equivalence \eqref{Joint_Pdm_Equiv} is given, for all $x \in \dR^2$, by $K(x) = J(A^{-1} x)$, that is
\begin{equation*}
K(x) = \frac{1}{2} x^{\prime} \Gamma^{-1} x,
\end{equation*}
where $\Gamma = \sigma^2 A \Lambda A^{\prime}$ is given by \eqref{Gamma}, and where we shall suppose that $\theta \neq -\rho$ to ensure that $A$ is invertible. In particular, the latter result also implies that the good rate function of the LDP on $\dR$ with speed $b_{n}^2$ associated with $(\sqrt{n} (\rn - \rho^{*})/b_{n})$ is given, for all $x \in \dR$, by
\begin{equation*}
I_{\rho}(x) = \frac{x^2}{2 \sigma^2_{\rho}},
\end{equation*}
where $\sigma^2_{\rho}$ is the last element of the matrix $\Gamma$. This achieves the proof of Theorem \ref{THM_PDMGAUSS_RHO}. \hfill
$\mathbin{\vbox{\hrule\hbox{\vrule height1ex \kern.5em\vrule height1ex}\hrule}}$

\subsection*{}
\begin{center}
{\bf 4.3. Proof of Theorem \ref{THM_PDMGAUSS_D}.}
\end{center}
For all $n \geq 1$, denote by $f_{n}$ the explosion coefficient associated with $J_{n}$ given by \eqref{J}, that is
\begin{equation}
\label{F}
f_{n} = \frac{J_{n} - J_{n-1}}{J_{n}} = \frac{\en^{\hspace{0.1cm} 2}}{J_{n}}.
\end{equation}
It follows from decomposition (C.4) in \cite{BercuProia11} that
\begin{equation}
\label{Decomp_D}
\frac{\sqrt{n}}{b_{n}} \left( \dn - D^{*} \right) = -2 \frac{\sqrt{n}}{b_{n}} \Big( 1 - f_{n} \Big) \Big( \rn - \rho^{*} \Big) + \frac{\sqrt{n}}{b_{n}} \zeta_{n},
\end{equation}
where the remainder term $\zeta_{n}$ is made of isolated terms. As before, we clearly have
\begin{equation*}
\frac{\sqrt{n}}{b_{n}} \zeta_{n} \superexp 0 \hspace{1cm} \text{and} \hspace{1cm} f_{n} \superexp 0.
\end{equation*}
As a consequence,
\begin{equation}
\label{D_Pdm_Equiv}
\frac{\sqrt{n}}{b_{n}} \left( \dn - D^{*} \right) \superexpequiv -2 \frac{\sqrt{n}}{b_{n}} \Big( \rn - \rho^{*} \Big),
\end{equation}
and this implies that both of them share the same LDP. The contraction principle \cite{DemboZeitouni98} enables us to conclude that the rate function of the LDP on $\dR$ with speed $b_{n}^2$ associated with equivalence \eqref{D_Pdm_Equiv} is given, for all $x \in \dR$, by $I_{D}(x) = I_{\rho}(-x/2)$, that is
\begin{equation*}
I_{D}(x) = \frac{x^2}{2 \sigma^2_{D}},
\end{equation*}
which achieves the proof of Theorem \ref{THM_PDMGAUSS_D}. \hfill
$\mathbin{\vbox{\hrule\hbox{\vrule height1ex \kern.5em\vrule height1ex}\hrule}}$

\subsection*{}
\begin{center}
{\bf 4.4. Proofs of Theorem \ref{THM_PDMCL_THETA}, Theorem \ref{THM_PDMCL_RHO} and Theorem \ref{THM_PDMCL_D}.}
\end{center}

We shall now propose a technical lemma ensuring that all results already proved under the gaussian assumption still hold under the Chen-Ledoux type condition.

\begin{lem}
\label{LEM_CVGEXP_CL}
Under \CL{HYP_CL}, \CL{HYP_E0} and \CL{HYP_X0}, all exponential convergences of Lemma \ref{LEM_CVGEXP_SN}, Corollary \ref{COR_CVGEXP_PN}, Corollary \ref{COR_CVGEXP_QN} and Lemma \ref{LEM_CVGEXP_AN} still hold.
\end{lem}
\begin{proof}
Under  \CL{HYP_CL}, \CL{HYP_E0} and \CL{HYP_X0}, and  following the same methodology as the one used to establish \eqref{Cvg_Pdm_Rn}, we get
\begin{equation}
\label{Cvg_Exp_Xn_CL}
 \frac{X_n^2}{b_{n} \sqrt{n}}\superexp 0,
\end{equation}
and Cauchy-Schwarz inequality implies that this is also the case for any isolated term of order 2, such as $X_{n} X_{n-1} /(b_{n} \sqrt{n})$. This allows us to control each remainder term. Note that \CL{HYP_E0}, \CL{HYP_X0} and \eqref{Cvg_Exp_Xn_CL} are obviously true for $\veps_0^4/n$, $X_0^4/n$, $\veps_0^2/n$, $X_0^2/n$ and $X_{n}^2/n$, since $b_{n} \sqrt{n} = o(n)$. Moreover, if follows from Theorem 2.2 of \cite{EichelsbacherLowe03} under \CL{HYP_CL} with $a = 2$, that
\begin{equation}
\label{Cvg_Exp_Ln_CL}
\frac{L_{n}}{n}\superexp \sigma^2.
\end{equation}
Furthermore, since $(M_n)$ is a locally square integrable martingale, we infer
from Theorem 2.1 of \cite{BercuTouati08} that for all $x, y >0$,
\begin{equation}
\label{BerCro}
\dP \Big( \vert M_{n} \vert > x, \langle M \rangle_{n} + [M]_{n} \leq y \Big) \leq 2 \exp\left( -\frac{x^2}{2y} \right),
\end{equation}
where the predictable quadratic variation $\langle M \rangle_{n} = \sigma^2 S_{n-1}$ is described in \eqref{S} and the total quadratic variation is given by $[M]_0=0$ and, for all $n \geq 1$, by
\begin{equation}
\label{Total_Var}
[M]_{n} = \sum_{k=1}^{n} X_{k-1}^2 V_{k}^2.
\end{equation}
According to (\ref{BerCro}), we have for all $\delta > 0$ and a suitable $b > 0$,
\begin{eqnarray*}
\dP \left( \frac{ \vert M_{n} \vert}{n} > \delta \right) & \leq & \dP\Big( \vert M_{n} \vert> \delta n,\langle M \rangle_{n} + [M]_{n} \leq n b \Big) + \dP\Big(\langle M \rangle_{n} + [M]_{n}> n b \Big), \\
&\leq & 2\exp\left(-\frac{n\delta^2}{2b} \right) + \dP\Big(\langle M \rangle_{n} + [M]_{n}> n b \Big), \\
& \leq & 2 \max\left( \dP\Big(\langle M \rangle_{n} + [M]_{n}> n b \Big) , 2\exp\left(-\frac{n\delta^2}{2b} \right)     \right).
\end{eqnarray*}
Consequently,
\begin{equation}
\label{Ineg_Mn_CL}
\limsup_{n \rightarrow \infty} \frac{1}{b_{n}^2} \log \dP\left( \frac{ \vert M_{n} \vert}{n} > \delta \right) \leq \limsup_{n \rightarrow \infty} \frac{1}{b_{n}^2} \log \dP\Big(\langle M \rangle_{n} + [M]_{n}> n b \Big). \vspace{0.3cm}
\end{equation}
We have for all $b>0$,
\begin{eqnarray}
\label{Ineq_TpsArret}
\dP\Big(\langle M \rangle_{n} + [M]_{n}> n b \Big) & \leq & \dP\bigg( \langle M \rangle_{n} >\frac{nb}{2} \bigg) + \dP\bigg( [M]_{n} > \frac{nb}{2} \bigg), \nonumber \\
& \leq & 2 \max\left( \dP\bigg( \langle M \rangle_{n} >\frac{nb}{2}\bigg), \dP\bigg( [M]_{n} > \frac{nb}{2}\bigg) \right).
\end{eqnarray}
Moreover, for all $n \geq 1$, let us define
\begin{equation*}
T_{n} = \sum_{k=0}^{n} X_{k}^4 \hspace{1cm} \text{and} \hspace{1cm} \Gamma_{n} = \sum_{k=1}^{n} V_{k}^4,
\end{equation*}
and note that we easily have the following inequality,
\begin{equation}
\label{Ineg_Tn}
T_{n} \leq \alpha X_0^4 + \beta \veps_0^4 + \beta \Gamma_{n} \hspace{0.5cm} \textnormal{a.s.}
\end{equation}
with $\alpha = 1 + (1 - \vert \theta \vert)^{-4}$ and $\beta = (1 - \vert \rho \vert)^{-4} (1 - \vert \theta \vert)^{-4}$. This implies that, for $n$ large enough, one can find $\gamma > 0$ such that
\begin{equation*}
T_{n} \leq \gamma \Gamma_{n} \hspace{0.5cm} \textnormal{a.s.}
\end{equation*}
choosing for example $\gamma = 3 \max(\alpha, \beta)$, under \CL{HYP_E0} and \CL{HYP_X0} for $a =4$. According to Theorem 2.2 of \cite{EichelsbacherLowe03} under \CL{HYP_CL} with $a = 4$, we also have the exponential convergence,
\begin{equation}
\label{Cvg_Exp_Gn_CL}
\frac{\Gamma_{n}}{n} \superexp \tau^4,
\end{equation}
where $\tau^4 = \dE[V_{1}^4]$, leading, via Cauchy-Schwarz inequality and \eqref{Ineg_Tn}, to
\begin{eqnarray}
\label{Cvg_Exp_Tn_CL}
\limsup_{n \rightarrow \infty} \frac{1}{b_{n}^2} \log \dP\left( \frac{[M]_{n}}{n} > \delta \right) & \leq & \limsup_{n \rightarrow \infty} \frac{1}{b_{n}^2} \log \dP\left( \frac{\Gamma_{n}}{n} > \frac{\delta}{\sqrt{\gamma}} \right), \nonumber \\
 & = & -\infty,
\end{eqnarray}
where $\delta > \tau^4 \sqrt{\gamma}$. Exploiting \eqref{Ineg_Sn} and \eqref{Cvg_Exp_Ln_CL}, the same result can be achieved for $\langle M \rangle_{n}/n$ under \CL{HYP_CL} with $a=2$ and $\delta > \sigma^4 \gamma$. As a consequence, it follows from \eqref{Ineq_TpsArret}, \eqref{Cvg_Exp_Tn_CL} and the latter remark that
\begin{eqnarray}
\label{Sinfini}
\limsup_{n \rightarrow \infty} \frac{1}{b_{n}^2} \log \dP\left( \frac{\langle M \rangle_{n} + [M]_{n}}{n} > b \right) = -\infty,
\end{eqnarray}
as soon as $b > \sigma^4 \gamma + \tau^4 \sqrt{\gamma}$. Therefore, the exponential convergence of $M_{n}/n$ to 0 with speed $b_{n}^2$ is obtained
via \eqref{Ineg_Mn_CL} and \eqref{Sinfini}, that is, for all $\delta >0$ and $b > \sigma^4 \gamma + \tau^4 \sqrt{\gamma}$,
\begin{equation}
\label{Cvg_Exp_Mn_CL}
\limsup_{n \rightarrow \infty} \frac{1}{b_{n}^2} \log \dP\left( \frac{|M_{n}|}{n} > \delta \right) = -\infty.
\end{equation}
The same obviously holds for $N_{n}/n$. Following the same lines as in the proofs of Lemma \ref{LEM_CVGEXP_SN}, Corollary \ref{COR_CVGEXP_PN}, Corollary \ref{COR_CVGEXP_QN} and Lemma \ref{LEM_CVGEXP_AN}, hypothesis \CL{HYP_E0} and \CL{HYP_X0} with $a=4$ together with exponential convergences \eqref{Cvg_Exp_Xn_CL}, \eqref{Cvg_Exp_Ln_CL} and \eqref{Cvg_Exp_Mn_CL} are sufficient to achieve the proof of Lemma \ref{LEM_CVGEXP_CL}.
\end{proof}
Let us introduce a simplified version of Puhalskii's result \cite{Puhalskii97} applied to a sequence of martingale differences, and two technical lemmas that shall help us to prove our results.
\begin{thm}[Puhalskii]
\label{THM_PUHALSKII_CL}
Let $(m_{j}^{n})_{1 \leq j \leq n}$ be a triangular array of martingale differences with values in $\dR^{d}$, with respect to the filtration $(\cF_{n})_{n \geq 1}$. Let $(b_{n})$ be a sequence of real numbers satisfying \eqref{BN_RATE}. Suppose that there exists a symmetric positive-semidefinite matrix $Q$ such that
\begin{equation}
\label{H1_Puhalskii}
\frac{1}{n} \sum_{k=1}^{n} \dE\Big[ m_{k}^{n} (m_{k}^{n})^{\prime} \big\vert \cF_{k-1} \Big] \superexp Q.
\end{equation}
Suppose that there exists a constant $c > 0$ such that, for each $1 \leq k \leq n$,
\begin{equation}
\label{H2_Puhalskii}
\vert m_{k}^{n} \vert \leq c \frac{\sqrt{n}}{b_{n}} \hspace{0.5cm} \textnormal{a.s.}
\end{equation}
Suppose also that, for all $a > 0$, we have the exponential Lindeberg's condition
\begin{equation}
\label{H3_Puhalskii}
\frac{1}{n} \sum_{k=1}^n \dE\Big[\vert m_{k}^{n} \vert^2 \mathrm{I}_{\left\{ \vert m_{k}^{n} \vert \geq a \frac{\sqrt{n}}{b_{n}} \right\}} \big\vert \cF_{k-1} \Big] \superexp 0.
\end{equation}
Then, the sequence
\begin{equation*}
\left( \frac{1}{b_{n} \sqrt{n}} \sum_{k=1}^{n} m_{k}^{n} \right)_{n \geq 1}
\end{equation*}
satisfies an LDP on $\dR^{d}$ with speed $b_{n}^2$ and good rate function
\begin{equation*}
\Lambda^*(v) = \sup_{\lambda \in \dR^{d}} \left( \lambda^{\prime} v - \frac{1}{2}\lambda^{\prime} Q \lambda\right).
\end{equation*}
In particular, if $Q$ is invertible,
\begin{equation}
\label{Pgd_Puhalskii_Rate_CL}
\Lambda^*(v) = \frac{1}{2} v^{\prime} Q^{-1} v.
\end{equation}
\end{thm}
\begin{proof}
The proof of Theorem \ref{THM_PUHALSKII_CL} is contained e.g. in the proof of Theorem 3.1 in \cite{Puhalskii97}.
\end{proof}
\begin{lem}
\label{LEM_LINDEXP_CL}
Under \CL{HYP_CL}, \CL{HYP_E0} and \CL{HYP_X0} with $a = 2$, we have for all $\delta > 0$,
\begin{equation*}
\limsup_{R \rightarrow \infty} \limsup_{n \rightarrow \infty} \frac{1}{n} \log \dP\left( \frac{1}{n} \sum_{k=1}^{n} X_{k}^2 \mathrm{I}_{\{ \vert X_{k} \vert > R \}} > \delta \right) < 0.
\end{equation*}
\end{lem}
\begin{rem}
\label{REM_LINDEXP_CL}
Lemma \ref{LEM_LINDEXP_CL} implies that the exponential Lindeberg's condition given by \eqref{H3_Puhalskii} is satisfied.
\end{rem}
\begin{proof}
We introduce the empirical measure associated with the geometric ergodic Markov
chain $(X_{n})_{n \geq 0}$,
\begin{equation}
\label{Mes_Emp}
\Lambda_{n} = \frac{1}{n} \sum_{k=1}^{n} \delta_{X_{k}},
\end{equation}
with invariant probability measure denoted by $\mu$. It is well-known that the sequence $(\Lambda_{n})$ satisfies the upper bound of the moderate deviations, see e.g. \cite{DjelloutGuillin01} for more details. Let us define, for $f(x) = x^2$, the following truncations,
\begin{equation*}
f^{(R)}(x) = f(x) \min \Big( 1, \big( f(x)-(R-1) \big)_{+} \Big) \hspace{0.5cm} \text{and} \hspace{0.5cm} \widetilde{f}^{(R)}(x) = \min\Big( f^{(R)}(x), R \Big).
\end{equation*}
Thus, we have
\begin{equation*}
0 \leq f(x) \mathrm{I}_{\{ f(x) \geq R \}} \leq f^{(R)}(x) \leq f(x),
\end{equation*}
and, as a consequence,
\begin{equation*}
0 \leq \Lambda_{n}\Big( f \mathrm{I}_{\{ f \geq R \}} \Big) \leq \Lambda_{n}\Big( f^{(R)}-\widetilde{f}^{(R)} \Big) + \Lambda_{n} \Big( \widetilde{f}^{(R)} \Big) - \mu\Big( \widetilde{f}^{(R)}\Big) + \mu\Big( \widetilde{f}^{(R)} \Big).
\end{equation*}
We also have
\begin{equation*}
f^{(R)} - \widetilde{f}^{(R)} = \Big( f^{(R)} - R \Big) \mathrm{I}_{\{ f^{(R)} \geq R \}} \leq \Big( f-R \Big) \mathrm{I}_{\{ f \geq R \}} = f- \Big( f \wedge R \Big).
\end{equation*}
For $\delta > 0$, the functions $\widetilde{f}^{(R)}$ and $f-(f \wedge R)$ are
continuous and bounded by $f$ which is $\mu$-integrable, and they converge to
0 as $R$ goes to infinity. By Lebesgue's Theorem, there exists $R > 0$
large enough such that $\mu(\widetilde{f}^{(R)}) + \mu(f- (f \wedge R)) < \delta/4.$
Thus,
\begin{eqnarray}
\label{Ineg_Tronc}
\dP\left( \frac{1}{n} \sum_{k=1}^{n} X_{k}^2 \mathrm{I}_{ \{ X_{k}^2 \geq R \}} > \delta \right) & \leq & \dP\Big( \Lambda_{n}(f) - \mu(f) > \delta/4 \Big) + \dP\Big( \Lambda_{n}(f \wedge R) - \mu(f \wedge R) > \delta/4 \Big) \nonumber\\
& & \hspace{0.5cm} + \hspace{0.2cm} \dP\Big( \Lambda_{n}(\widetilde{f}^{(R)}) - \mu(\widetilde{f}^{(R)}) > \delta/4 \Big).
\end{eqnarray}
From Lemma \ref{LEM_CVGEXP_CL}, we have that for all $\delta > 0$,
\begin{equation*}
\limsup_{n \rightarrow \infty} \frac{1}{b_{n}^2} \log \dP\Big( \Lambda_{n}(f) - \mu(f) > \delta \Big) = -\infty.
\end{equation*}
By the upper bound of the moderate deviation principle for the sequence $(\Lambda_{n})$ given in \cite{DjelloutGuillin01}, we obtain that
\begin{equation*}
\limsup_{R \rightarrow \infty} \limsup_{n \rightarrow \infty} \frac{1}{b_{n}^2} \log\dP\Big( \Lambda_{n} (f \wedge R) - \mu(f \wedge R) > \delta \Big) = -\infty,
\end{equation*}
and
\begin{equation*}
\limsup_{R \rightarrow \infty} \limsup_{n \rightarrow \infty} \frac{1}{b_{n}^2} \log\dP \Big( \Lambda_{n}(\widetilde{f}^{(R)}) - \mu(\widetilde{f}^{(R)}) > \delta \Big) = -\infty,
\end{equation*}
which, via inequality \eqref{Ineg_Tronc}, achieves the proof of Lemma \ref{LEM_LINDEXP_CL}. Note that Remark \ref{REM_LINDEXP_CL} is immediately derived from the latter proof, see e.g. \cite{Worms00} for more details.
\end{proof}
\begin{lem}
\label{LEM_PDM_MN_CL}
Under \CL{HYP_CL}, \CL{HYP_E0} and \CL{HYP_X0}, the sequence
\begin{equation*}
\left( \frac{M_{n}}{b_{n} \sqrt{n}} \right)_{n \geq 1}
\end{equation*}
satisfies an LDP on $\dR$ with speed $b_{n}^2$ and good rate function
\begin{equation}
\label{Pgd_M_Rate_CL}
J(x) = \frac{x^2}{2 \ell \sigma^2}
\end{equation}
where $\ell$ is given by \eqref{LimS}.
\end{lem}
\begin{proof}
From now on, in order to apply Puhalskii's result for the moderate deviations for martingales, we introduce the following modification of the martingale $(M_n)_{n\ge 0}$, for $r>0$ and $R>0$,
\begin{equation}
\label{Mtrunc}
M_n^{(r,R)}=\sum_{k=1}^nX_{k-1}^{(r)}V_k^{(R)}
\end{equation}
where, for all $1 \leq k \leq n$,
\begin{equation}
\label{XVtrunc} X_k^{(r)}=X_k\mathrm{I}_{\big\{ |X_k| \leq r\frac{\sqrt n}{b_n} \big\}} \hspace{0.5cm} \text{and} \hspace{0.5cm} V_k^{(R)}=V_k\mathrm{I}_{\big\{ |V_k|\leq R \big\}}-\dE\Big[ V_k\mathrm{I}_{\big\{|V_k|\leq R \big\}} \Big].
\end{equation}
Then, we have to prove that for all $r>0$ the sequence $(M_n^{(r,R)})$ is an exponentially good approximation of $(M_n)$ as $R$ goes to infinity, see e.g. Definition 4.2.14 in \cite{DemboZeitouni98}. This approximation, in the sense of the large deviations, is described by the following convergence, for all $r> 0$ and all $\delta>0$,
\begin{equation}
\label{cont} \limsup_{R\rightarrow \infty}\limsup_{n\rightarrow \infty}\frac{1}{b_n^2}\log\dP\left(\frac{\vert M_n-M_n^{(r,R)} \vert}{b_n\sqrt n} > \delta\right)=-\infty.
\end{equation}
From Lemma \ref{LEM_CVGEXP_CL}, and since $\langle M \rangle_n = \sigma^2S_{n-1}$, we have
\begin{equation}
\label{croch}
\frac{\langle M \rangle_n}{n} \superexp \sigma^2 \ell.
\end{equation}
From Lemma \ref{LEM_CVGEXP_CL} and Remark \ref{REM_LINDEXP_CL}, we also have for all $r>0$,
\begin{equation}
\label{Linder} \frac{1}{n}\sum_{k=0}^{n} X_{k}^2 \mathrm{I}_{\left\{ |X_{k}| > r\frac{\sqrt n}{b_n} \right\}} \superexp 0.
\end{equation}
We introduce the following notations,
\begin{equation*}
\sigma_R^2=\dE \left[ ( V_1^{(R)} )^2 \right] \qquad {\rm and} \qquad S_n^{(r)} = \sum_{k=0}^n ( X_k^{(r)} )^2.
\end{equation*}
Then, we easily transfer properties $\eqref{croch}$ and $\eqref{Linder}$ to the truncated martingale $(M_n^{(r,R)})_{n\ge 0}$. We have for all $R>0$ and all $r>0$,
\begin{equation*}
\frac{\langle M^{(r,R)} \rangle_n}{n} = \sigma_R^2\frac{S_{n-1}^{(r)}}{n}=-\sigma_R^2\left( \frac{S_{n-1}}{n}-\frac{S_{n-1}^{(r)}}{n}\right)+\sigma_R^2\frac{S_{n-1}}{n} \superexp \sigma_R^2 \ell
\end{equation*}
which ensures that \eqref{H1_Puhalskii} is satisfied for the martingale $(M_n^{(r,R)})_{n\ge 0}$. Note also that Lemma \ref{LEM_CVGEXP_CL} and Remark \ref{REM_LINDEXP_CL} work for the martinagle $(M_n^{(r,R)})_{n\ge 0}$. So, for all $r>0$, the exponential Lindeberg's condition and thus \eqref{H3_Puhalskii} are satisfied for $(M_n^{(r,R)})_{n\ge 0}$.
By Theorem \ref{THM_PUHALSKII_CL}, we deduce that $(M_n^{(r,R)}/b_n\sqrt{n})$ satisfies an LDP on $\dR$ with speed $b_n^2$ and good rate function
\begin{equation}\label{rateIR}J_R(x)=\frac{x^2}{2\sigma_R^2\ell }.\end{equation}
It will be possible to drive the moderate deviations result for the martingale $(M_n)_{n\ge 0}$ by proving relation $\eqref{cont}$. For that matter, let us now introduce the following decomposition,
\begin{equation*}
M_n-M_n^{(r,R)}=L_n^{(r)}+F_n^{(r,R)}
\end{equation*}
where
\begin{equation*}
L_n^{(r)}=\sum_{k=1}^n \Big( X_{k-1}-X_{k-1}^{(r)} \Big)V_k \hspace{0.5cm} \text{and} \hspace{0.5cm} F_n^{(r,R)}=\sum_{k=1}^n \Big( V_k-V^{(R)}_k \Big) X_{k-1}^{(r)}.
\end{equation*}
One has to show that for all $r>0$,
\begin{equation}
\label{Lr}
\frac{L_n^{(r)}}{b_n\sqrt n} \superexp 0,
\end{equation}
and, for all $r>0$ and all $\delta>0$, that
\begin{equation}
\label{FR}
\limsup_{R\rightarrow \infty}\limsup_{n\rightarrow \infty}\frac{1}{b_n^2}\log\dP\left(\frac{|F_n^{(r,R)}|}{b_n\sqrt n}>\delta\right)= -\infty.
\end{equation}
On the one hand,
note that for any $\eta > 0$,
\begin{equation*}
\sum_{k=0}^{n} \vert X_{k} \vert^{2+\eta} \leq \alpha \vert X_0 \vert^{2+\eta} + \beta \vert \veps_0 \vert^{2+\eta} + \beta \sum_{k=1}^{n} \vert V_{k} \vert^{2+\eta} \hspace{0.5cm} \textnormal{a.s.}
\end{equation*}
with $\alpha = 1 + (1 - \vert \theta \vert)^{-(2+\eta)}$ and $\beta = (1 - \vert \rho \vert)^{-(2+\eta)} (1 - \vert \theta \vert)^{-(2+\eta)}$. This implies that, for $n$ large enough, one can find $\gamma > 0$ such that
\begin{equation}
\label{Ineg_XnVn}
\sum_{k=0}^{n} \vert X_{k} \vert^{2+\eta} \leq \gamma \sum_{k=1}^{n} \vert V_{k} \vert^{2+\eta} \hspace{0.5cm} \textnormal{a.s.}
\end{equation}
taking for example $\gamma = 3 \max(\alpha, \beta)$, under \CL{HYP_E0} and \CL{HYP_X0} for $a = 2+\eta$.
Thus,
\begin{eqnarray}
\label{Ineq_Ln_Tronc}
\frac{\vert L_n^{(r)} \vert}{b_{n}\sqrt{n}} & = & \frac{1}{b_{n}\sqrt{n}} \left\vert \sum_{k=1}^{n} X_{k-1} \mathrm{I}_{\big\{ |X_{k-1}| > r\frac{\sqrt n}{b_n} \big\}} V_{k} \right\vert, \nonumber \\
 & \leq & \frac{1}{b_{n}\sqrt{n}} \left( r \frac{\sqrt{n}}{b_{n}} \right)^{\! -\eta} \left( \sum_{k=1}^{n} \vert X_{k-1} \vert^{2+\eta} \right)^{\! 1/2} \left( \sum_{k=1}^{n} V_{k}^{2} \vert X_{k-1} \vert^{\, \eta} \right)^{\! 1/2}, \nonumber \\
 & \leq & \lambda(r,\eta,\gamma) \left( \frac{b_{n}}{\sqrt{n}} \right)^{\! \eta-1} \frac{1}{n} \sum_{k=1}^{n} \vert V_{k} \vert^{2+\eta} \hspace{0.5cm} \textnormal{a.s.}
\end{eqnarray}
by virtue of \eqref{Ineg_XnVn} and H\"older's inequality, where $\lambda(r,\eta,\gamma) > 0$ can be evaluated under suitable assumptions of moment on $(V_{n})$. As a consequence, for all $\delta > 0$,
\begin{eqnarray}
\label{Cvg_Ln_1}
\limsup_{n \rightarrow \infty} \frac{1}{b_{n}^2} \log\dP\left( \frac{\vert L_n^{(r)} \vert}{b_{n}\sqrt{n}} > \delta \right) & \leq & \limsup_{n \rightarrow \infty} \frac{1}{b_{n}^2} \log\dP\left( \frac{1}{n} \sum_{k=1}^{n} \vert V_{k} \vert^{2+\eta} > \frac{\delta}{\lambda(r,\eta,\gamma)} \left( \frac{\sqrt{n}}{b_{n}} \right)^{\! \eta-1} \right), \nonumber \\
 & = & -\infty,
\end{eqnarray}
as soon as $\eta > 1$, under \CL{HYP_CL} with $a = 2+\eta$. We deduce that
\begin{equation}
\label{Cvg_Exp_Lnr}
\frac{L_n^{(r)}}{b_{n}\sqrt{n}} \superexp 0,
\end{equation}
which achieves the proof of \eqref{Lr}, under \CL{HYP_CL}, \CL{HYP_E0} and \CL{HYP_X0} for $a > 3$. On the other hand, $(F^{(r,R)}_n)_{n \geq 0}$ is a locally square-integrable real martingale whose predictable quadratic variation is given by $\langle F^{(r,R)} \rangle_0 = 0$ and, for all $n \geq 1$, by
\begin{equation*}
\langle F^{(r,R)} \rangle_n = \dE\left[\left(V_{1} - V_{1}^{(R)}\right)^{2}\right] S_{n-1}^{(r)}.
\end{equation*}
To prove (\ref{FR}), we will use Theorem 1 of \cite{Djellout02}. For $R$ large enough and all $k \geq 1$, we have
\begin{eqnarray*}
\dP\Bigg(\left\vert X_{k-1}^{(r)} \left(V_{k} - V_{k}^{(R)} \right) \right\vert
> b_{n}\sqrt{n} ~ \Big\vert \cF_{k-1}\Bigg) & \leq & \dP\left(\left\vert V_{k} -
V_{k}^{(R)} \right\vert > \frac{b_{n}^{2}}{r} \right),\\
 & = & \dP\left( \left\vert V_{1} -
V_{1}^{(R)} \right\vert > \frac{b_{n}^{2}}{r}\right) = 0.
\end{eqnarray*}
This implies that
\begin{equation}
\label{EssSup_Fn}
\limsup_{n\rightarrow \infty} \frac{1}{b_{n}^{2}} \log \left( n\,\,\,
\underset{k \geq 1} {\rm ess\,sup} ~
\dP\Bigg(\left\vert X_{k-1}^{(r)} \left(V_{k} - V_{k}^{(R)} \right) \right\vert
> b_{n}\sqrt{n} ~ \Big\vert \cF_{k-1}\Bigg) \right) = -\infty.
\end{equation}
For all $\gamma > 0$ and all $\delta > 0$, we obtain from
Lemma \ref{LEM_LINDEXP_CL} and Remark \ref{REM_LINDEXP_CL}, that
\begin{eqnarray*}
\limsup_{n\rightarrow \infty} \frac{1}{b_{n}^{2}}
\log \dP\left( \frac{1}{n} \sum_{k=1}^{n} \left( X_{k-1}^{(r)} \right)^{2} \mathrm{I}_{\left\{ \vert X_{k-1}^{(r)} \vert > \gamma \frac{\sqrt{n}}{b_{n}}\right\}} > \delta\right) \leq \hspace{5cm} \\
\hspace{5cm} \limsup_{n\rightarrow \infty} \frac{1}{b_{n}^{2}} \log \dP\left( \frac{1}{n} \sum_{k=1}^{n} X_{k-1}^{2} \mathrm{I}_{\left\{\vert X_{k-1} \vert > \gamma \frac{\sqrt{n}}{b_{n}}\right\}} > \delta\right) = -\infty.
\end{eqnarray*}
Finally, from Lemma \ref{LEM_CVGEXP_CL}, Lemma \ref{LEM_LINDEXP_CL} 
and Remark \ref{REM_LINDEXP_CL}, it follows that
\begin{equation*}
\frac{\langle F^{(r,R)} \rangle_n}{n} =
Q_R\frac{S_{n-1}^{(r)}}{n}=-Q_R\left(
\frac{S_{n-1}}{n}-\frac{S_{n-1}^{(r)}}{n}\right)+
Q_R\frac{S_{n-1}}{n} \superexp Q_R \ell
\end{equation*}
where
\begin{equation*}
Q_{R} = \dE\left[\left(V_{1} - V_{1}^{(R)}\right)^{2}\right],
\end{equation*}
and $\ell$ is given by \eqref{LimS}. Moreover, it is clear that $Q_{R}$ converges to 0 as $R$ goes to infinity. In light of foregoing, we infer from Theorem 1 of \cite{Djellout02} that
$(F_n^{(r,R)}/(b_{n}\sqrt{n}))$ satisfies an LDP on $\dR$ of speed $b_{n}^{2}$ and good rate function
\begin{equation*}
I_{R}(x) = \frac{x^{2}}{2 Q_{R} \ell}.
\end{equation*}
In particular, this implies that for all $\delta > 0$,
\begin{equation}
\label{Pdm_F}
\limsup\limits_{n\rightarrow\infty} \frac{1}{b_{n}^{2}}
\log\dP\left( \frac{|F_n^{(r,R)}|}{b_{n}\sqrt{n}}>\delta\right) \leq
-\frac{\delta^{2}}{2Q_{R}\ell},
\end{equation}
and letting $R$ go to infinity clearly leads to the end of the proof of \eqref{FR}. We are able to conclude now that  $(M_n^{(r,R)}/ (b_n \sqrt{n}))$ is an exponentially good approximation of $(M_n/(b_n\sqrt{n}))$. By application of  Theorem 4.2.16 in \cite{DemboZeitouni98}, we find that $(M_{n}/(b_n\sqrt{n}))$ satisfies an LDP on $\dR$ with speed $b_n^2$ and good rate function
\begin{equation*}
\widetilde{J}(x) = \sup_{\delta >0}\liminf_{R\rightarrow \infty}\inf_{z\in B_{x,\delta}}J_R(z),
\end{equation*}
where $J_R$ is given in \eqref{rateIR}  and $B_{x,\delta}$ denotes the ball $\{z:|z-x|<\delta\}.$ The identification of the rate function $\widetilde J= J$, where $J$ is given in \eqref{Pgd_M_Rate_CL} is done easily, which concludes the proof of Lemma \ref{LEM_PDM_MN_CL}.
\end{proof}

\begin{rem}
If we suppose that \CL{HYP_CL} holds with $a > 2$, then the
exponential Lindeberg's condition in Lemma \ref{LEM_LINDEXP_CL} is easier to establish. Indeed, using (\ref{Ineg_XnVn}), it follows that
\begin{equation*}
\left(r \frac{\sqrt n}{b_n}\right)^{\! \eta} \sum_{k=1}^{n} X_{k-1}^2 \mathrm{I}_{\left\{ \vert X_{k-1} \vert >  r\frac{\sqrt n}{b_n}\right\}} \leq  \sum_{k=1}^{n} |X_{k-1}|^{2+\eta} \leq \gamma \sum_{k=1}^{n} \vert V_{k} \vert^{2+\eta},
\end{equation*}
for $n$ large enough and $\eta > 0$, leading to
\begin{equation*}
\dP\left(\frac{1}{n}\sum_{k=1}^{n} X_{k-1}^2 \mathrm{I}_{\left\{ \vert X_{k-1} \vert > r\frac{\sqrt n}{b_n}\right\}} >\delta\right)\leq \dP\left(\frac{1}{n}\sum_{k=1}^{n}
\vert V_{k} \vert^{2+\eta}>\frac{\delta}{\gamma} \left(r \frac{\sqrt n}{b_n}\right)^{\! \eta}\right).
\end{equation*}
\end{rem}

\vspace{0.3cm}

\begin{lem}
\label{LEM_PDM_MNVect_CL}
Under \CL{HYP_CL}, \CL{HYP_E0} and \CL{HYP_X0}, the sequence
\begin{equation*}
\left( \frac{1}{b_{n} \sqrt{n}} \begin{pmatrix}
M_{n} \\
N_{n}
\end{pmatrix} \right)_{n \geq 1}
\end{equation*}
satisfies an LDP on $\dR^2$ with speed $b_{n}^2$ and good rate function
\begin{equation}
\label{Pgd_MVect_Rate_CL}
J(x) = \frac{1}{2 \sigma^2} x^{\prime} \Lambda^{-1} x
\end{equation}
where $\Lambda$ is given by \eqref{Lam}.
\end{lem}
\begin{proof} We follow the same approach as in the proof of Lemma \ref{LEM_PDM_MN_CL}.
We shall consider the 2-dimensional vector martingale $(Z_{n})_{n \geq 0}$ defined in \eqref{Z}. In order to apply Theorem \ref{THM_PUHALSKII_CL}, we introduce the following truncation of the martingale $(Z_n)_{n \ge 0}$, for $r>0$ and $R>0$,
\begin{equation*}
Z^{(r,R)}_n=\begin{pmatrix}
M^{(r,R)}_n\\
N^{(r,R)}_n
\end{pmatrix}
\end{equation*}
where $M_n^{(r,R)}$ is given in \eqref{Mtrunc} and where $N_n^{(r,R)}$ is defined in the same manner, that is, for all $n \geq 2$,
\begin{equation}
\label{Ntrunc} N_n^{(r,R)} = \sum_{k=2}^n X_{k-2}^{(r)}V_k^{(R)}
\end{equation}
with $X_n^{(r)}$ and  $V_n^{(R)}$ given by $(\ref{XVtrunc})$. The exponential convergence \eqref{Cvg_Exp_Proc_Joint_Gauss} still holds, by virtue of Lemma \ref{LEM_CVGEXP_CL}, which immediately implies hypothesis \eqref{H1_Puhalskii}. On top of that, Lemma \ref{LEM_LINDEXP_CL} ensures that, for all $r > 0$,
\begin{equation}
\label{linder}
\frac{1}{n}\sum_{k=0}^nX_{k}^2 \mathrm{I}_{\left\{ \vert X_{k} \vert > r\frac{\sqrt n}{b_n} \right\} } \superexp 0,
\end{equation}
justifying hypothesis \eqref{H3_Puhalskii}. Via Theorem \ref{THM_PUHALSKII_CL}, $(Z_n^{(r,R)}/(b_n\sqrt{n}))$ satisfies an LDP on $\dR^2$ with speed $b_{n}^2$ and good rate function $J_R$ given by
\begin{equation}
\label{tauxJR}
J_R(x) = \frac{1}{2 \sigma_R^2 } x^{\prime} \Lambda^{-1} x.
\end{equation}
Finally, it is straightforward to prove that $(Z_{n}^{(r,R)}/(b_{n} \sqrt{n}))$ is an exponentially good approximation of $(Z_{n}/(b_{n} \sqrt{n}))$. By application of Theorem 4.2.16 in \cite{DemboZeitouni98}, we deduce that $(Z_{n}/(b_{n} \sqrt{n}))$ satisfies an LDP on $\dR^2$ with speed $b_n^2$ and good rate function given by
\begin{equation*}
\widetilde{J}(x) = \sup_{\delta >0}\liminf_{R\rightarrow \infty}\inf_{z\in B_{x,\delta}}J_R(z),
\end{equation*}
 where $J_R$ is given in \eqref{tauxJR} and $B_{x,\delta}$ denotes the ball $\{z:|z-x|<\delta\}.$ The identification of the rate function $\widetilde J=J$ is done easily, which concludes the proof of Lemma \ref{LEM_PDM_MNVect_CL}.
\end{proof}

\medskip

\noindent{\bf Proofs of Theorem \ref{THM_PDMCL_THETA}, Theorem \ref{THM_PDMCL_RHO} and Theorem \ref{THM_PDMCL_D}.} The residuals appearing in the decompositions \eqref{Decomp_Theta}, \eqref{Decomp_Joint} and \eqref{Decomp_D} still converge exponentially to zero under \CL{HYP_CL}, \CL{HYP_E0} and \CL{HYP_X0}, with speed $b_{n}^2$, as it was already proved. Therefore, for a better readability, we may skip the most accessible parts of these proofs whose development merely consists in following the same lines as those in the proofs of Theorem \ref{THM_PDMGAUSS_THETA}, Theorem \ref{THM_PDMGAUSS_RHO} and Theorem \ref{THM_PDMGAUSS_D}, taking advantage of Lemma \ref{LEM_PDM_MN_CL} and Lemma \ref{LEM_PDM_MNVect_CL}, and applying the contraction principle given e.g. in \cite{DemboZeitouni98}. \hfill
$\mathbin{\vbox{\hrule\hbox{\vrule height1ex \kern.5em\vrule height1ex}\hrule}}$

\medskip

\noindent{\bf Acknowledgments.} \textit{The authors thank Bernard Bercu and Arnaud Guillin for all their advices and suggestions during the preparation of this work.}

\nocite{*}

\bibliographystyle{acm}
\bibliography{DMDW}

\begin{thebibliography}{10}

\bibitem{Arcones01}
{\sc Arcones, M.~A.}
\newblock The large deviation principle of empirical processes.
\newblock {\em Preprint\/} (2001).

\bibitem{BercuProia11}
{\sc Bercu, B., and Pro\"ia, F.}
\newblock A sharp analysis on the asymptotic behavior of the {D}urbin-{W}atson
  statistic for the first-order autoregressive process.
\newblock {\em arXiv 1104.3328v1. In revision. ESAIM Probab. Stat.\/} (2011).

\bibitem{BercuTouati08}
{\sc Bercu, B., and Touati, A.}
\newblock Exponential inequalities for self-normalized martingales with
  applications.
\newblock {\em Ann. Appl. Probab. 18, no.5\/} (2008), 1848--1869.

\bibitem{Chen98}
{\sc Chen, X.}
\newblock Moderate deviations for $m$-dependent random variables with {B}anach
  space value.
\newblock {\em Statis. and Probab. Letters. 35\/} (1998), 123--134.

\bibitem{Dembo96}
{\sc Dembo, A.}
\newblock Moderate deviations for martingales with bounded jumps.
\newblock {\em Electron. Comm. Probab. 1, no. 3\/} (1996), 11--17.

\bibitem{DemboZeitouni98}
{\sc Dembo, A., and Zeitouni, O.}
\newblock {\em Large deviations techniques and applications, second edition},
  vol.~38 of {\em Applications of Mathematics}.
\newblock Springer, 1998.

\bibitem{Djellout02}
{\sc Djellout, H.}
\newblock Moderate deviations for martingale differences and applications to
  $\phi$-mixing sequences.
\newblock {\em Stoch. Stoch. Rep. 73, 1-2\/} (2002), 37--63.

\bibitem{DjelloutGuillin01}
{\sc Djellout, H., and Guillin, A.}
\newblock Moderate deviations for {M}arkov chains with atom.
\newblock {\em Stochastic Process. Appl. 95, no. 2\/} (2001), 203--217.

\bibitem{Durbin70}
{\sc Durbin, J.}
\newblock Testing for serial correlation in least-squares regression when some
  of the regressors are lagged dependent variables.
\newblock {\em Econometrica 38\/} (1970), 410--421.

\bibitem{DurbinWatson50}
{\sc Durbin, J., and Watson, G.~S.}
\newblock Testing for serial correlation in least squares regression. {I}.
\newblock {\em Biometrika 37\/} (1950), 409--428.

\bibitem{DurbinWatson51}
{\sc Durbin, J., and Watson, G.~S.}
\newblock Testing for serial correlation in least squares regression. {II}.
\newblock {\em Biometrika 38\/} (1951), 159--178.

\bibitem{DurbinWatson71}
{\sc Durbin, J., and Watson, G.~S.}
\newblock Testing for serial correlation in least squares regession. {III}.
\newblock {\em Biometrika 58\/} (1971), 1--19.

\bibitem{EichelsbacherLowe03}
{\sc Eichelsbacher, P., and L\"owe, M.}
\newblock Moderate deviations for i.i.d. random variables.
\newblock {\em ESAIM Probab. Stat. 7\/} (2003), 209--218.

\bibitem{Inder86}
{\sc Inder, B.~A.}
\newblock An approximation to the null distribution of the {D}urbin-{W}atson
  statistic in models containing lagged dependent variables.
\newblock {\em Econometric Theory 2\/} (1986), 413--428.

\bibitem{KingWu91}
{\sc King, M.~L., and Wu, P.~X.}
\newblock Small-disturbance asymptotics and the {D}urbin-{W}atson and related
  tests in the dynamic regression model.
\newblock {\em J. Econometrics 47\/} (1991), 145--152.

\bibitem{Ledoux92}
{\sc Ledoux, M.}
\newblock Sur les d\'eviations mod\'er\'ees des sommes de variables
  al\'eatoires vectorielles ind\'ependantes de m\^eme loi.
\newblock {\em Ann. Inst. Henri-Poincar\'e. 35\/} (1992), 123--134.

\bibitem{Malinvaud61}
{\sc Malinvaud, E.}
\newblock Estimation et pr\'evision dans les mod\`eles \'economiques
  autor\'egressifs.
\newblock {\em Review of the International Institute of Statistics 29\/}
  (1961), 1--32.

\bibitem{NerloveWallis66}
{\sc Nerlove, M., and Wallis, K.~F.}
\newblock Use of the {D}urbin-{W}atson statistic in inappropriate situations.
\newblock {\em Econometrica 34\/} (1966), 235--238.

\bibitem{Puhalskii97}
{\sc Puhalskii, A.}
\newblock Large deviations of semimartingales: a maxingale problem approach.
  {I}. {L}imits as solutions to a maxingale problem.
\newblock {\em Stoch. Stoch. Rep. 61\/} (1997, no. 3-4), 141--243.

\bibitem{Stocker06}
{\sc Stocker, T.}
\newblock On the asymptotic bias of {OLS} in dynamic regression models with
  autocorrelated errors.
\newblock {\em Statist. Papers 48\/} (2007), 81--93.

\bibitem{Worms99}
{\sc Worms, J.}
\newblock Moderate deviations for stable {M}arkov chains and regression models.
\newblock {\em Electron. J. Probab. 4, no. 8\/} (1999), 1--28.

\bibitem{Worms00}
{\sc Worms, J.}
\newblock Principes de d\'eviations mod\'er\'ees pour des martingales et
  applications statistiques.
\newblock {\em Th\`ese de Doctorat \`a l'Universit\'e Marne-la-Vall\'ee.\/}
  (2000).

\bibitem{Worms01a}
{\sc Worms, J.}
\newblock Moderate deviations of some dependent variables. {I}. {M}artingales.
\newblock {\em Math. Methods Statist. 10, no. 1\/} (2001), 38--72.

\bibitem{Worms01b}
{\sc Worms, J.}
\newblock Moderate deviations of some dependent variables. {II}. {S}ome kernel
  estimators.
\newblock {\em Math. Methods Statist. 10, no. 2\/} (2001), 161--193.

\end{thebibliography}

\vspace{10pt}

\end{document}